\documentclass[10pt]{amsart}
\usepackage{amsxtra, amsfonts, amsmath, amsthm, amstext, amssymb, amscd, mathrsfs, verbatim, color}
\usepackage{threeparttable}
\usepackage{mathtools}
\usepackage[ansinew]{inputenc}\usepackage[T1]{fontenc}
\usepackage[all,cmtip]{xy}
\theoremstyle{plain}

\newtheorem{theorem}{Theorem}[section]
\newtheorem{lemma}[theorem]{Lemma}
\newtheorem{definition-theorem}[theorem]{Definition-Theorem}
\newtheorem{proposition}[theorem]{Proposition}
\newtheorem{corollary}[theorem]{Corollary}

\theoremstyle{definition}
\newtheorem{definition}[theorem]{Definition}
\newtheorem{example}[theorem]{Example}
\newtheorem{remark}[theorem]{Remark}
\newtheorem{notation}[theorem]{Notation}
\newcommand \bth[1] { \begin{theorem}\label{t#1} }
\newcommand \ble[1] { \begin{lemma}\label{l#1} }

\newcommand \bpr[1] { \begin{proposition}\label{p#1} }
\newcommand \bco[1] { \begin{corollary}\label{c#1} }
\newcommand \bde[1] { \begin{definition}\label{d#1}\rm }
\newcommand \bex[1] { \begin{example}\label{e#1}\rm }
\newcommand \bre[1] { \begin{remark}\label{r#1}\rm }

\newcommand \bnota[1] {\begin{notation}\label{n#1}\rm }
\newcommand {\ele} { \end{lemma} }

\newcommand {\epr} { \end{proposition} }
\newcommand {\eco} { \end{corollary} }
\newcommand {\ede} { \end{definition} }
\newcommand {\eex} { \end{example} }
\newcommand {\ere} { \end{remark} }
\newcommand {\enota} { \end{notation} }

\begin{document}
\title[Geometric lemma]{Dual of the Geometric Lemma and the Second Adjointness Theorem for $p$-adic reductive groups} 

\author[Kei Yuen Chan]{Kei Yuen Chan}
\address{Department of Mathematics, the University of Hong Kong \\}
\email{kychan1@hku.hk}
\maketitle

\begin{abstract}
Let $P,Q$ be standard parabolic subgroups of a $p$-adic reductive group $G$. We investigate the smooth dual of the filtration on a parabolically induced module arising from the geometric lemma associated with the cosets $P\setminus G/Q$. We prove that the dual filtration coincides with the filtration associated with the cosets $P\setminus G/Q^-$ via the Bernstein-Casselman canonical pairing from the second adjointness of parabolic induction. This result generalizes a result of Bezrukavnikov-Kazhdan on an explicit description of maps in the second adjointness. Our main input of the proof is group theoretic results on Richardson varieties. We also discuss some applications to Bernstein-Zelevinsky derivatives.
\end{abstract}

\section{Introduction}

Let \( G \) be a connected reductive group over a non-Archimedean local field \( F \). All the representations of \( G \) that we consider in this article are smooth and over \( \mathbb{C} \), unless otherwise specified. The parabolic induction and Jacquet functor form a pair of adjoint functors and are fundamental tools in the representation theory of reductive \( p \)-adic groups. The Bernstein geometric lemma describes how to decompose the Jacquet functors of parabolically induced modules, which is analogous to the Mackey decomposition in finite groups.

It follows from definitions that the parabolic induction has the Jacquet functor as its natural left adjoint. Casselman (see \cite{Ca93}) discovered that the opposite Jacquet functor is the right adjoint to parabolic induction for admissible representations. Extending this to all smooth representations is highly non-trivial, first demonstrated by Bernstein \cite{Be92} using his stabilization theorem. This later adjointness is commonly referred to as the second adjointness theorem. The geometric perspective of the second adjointness is explored in \cite{BK15}.

This paper studies interplay between the second adjointness and the geometric lemma. To state results more precisely, we need more notations. Let $P$ be a standard parabolic subgroup of $G$ with the Levi decomposition $MU$, where $M$ is the Levi subgroup of $P$ and $U$ is the unipotent radical of $P$. For a smooth representation $\sigma$ of $M$, define $\mathrm{Ind}_P^G\sigma$ to be the normalized parabolic induction; and for a smooth representation $\pi$ of $G$, define $\pi_U$ to be the normalized Jacquet functor. 

Let $Q$ be another standard parabolic subgroup of $G$ with the Levi decomposition $NV$, where $N$ is the Levi subgroup of $Q$ while $V$ is the unipotent subgroup of $Q$. The geometric lemma describes a decomposition for 
\[   (\mathrm{Ind}_P^G\sigma)_V
\]
as an $N$-representation. For a smooth representation $\pi$ of $G$, we denote by $\pi^{\vee}$ the smooth dual of $\pi$. By taking the exact contravariant dual smooth functor (as $N$-representations), we then obtain one filtration on $((\mathrm{Ind}_P^G\sigma)_V)^{\vee}$. 

Let $Q^-$ be the parabolic subgroup opposite to $Q$ and let $V^-$ be the  unipotent radical of $Q^-$. The second adjointness provides a natural isomorphism:
\[                 ((\mathrm{Ind}_P^G\sigma)_V)^{\vee} \cong ((\mathrm{Ind}_P^G\sigma)^{\vee})_{V^-} .
\]
The work of Bernstein and Casselman shows a more explicit non-degenerate $N$-equivariant pairing for $(\mathrm{Ind}_P^G\sigma)_V$ and $((\mathrm{Ind}_P^G\sigma)^{\vee})_{V^-}$, which we shall refer to as the Bernstein-Casselman canonical pairing. There is a natural isomorphism between $(\mathrm{Ind}_P^G\sigma)^{\vee}$ and $\mathrm{Ind}_P^G(\sigma^{\vee})$, allowing us to obtain a filtration on the space $((\mathrm{Ind}_P^G\sigma)^{\vee})_{V^-}=(\mathrm{Ind}_P^G(\sigma^{\vee}))_{V^-}$, using the orbits in $P\setminus G/Q^-$.

More precise definitions for the filtrations are provided in Sections \ref{ss geometric lemma} through \ref{ss dual geometric lemma}. We note that to match the two filtrations, the compatibility of the Bruhat orderings discussed in Section \ref{s weyl groups} is essential. Our main result is the following:

\begin{theorem} \label{thm coincide filtrations}
The two filtrations above coincide via the Bernstein-Casselman canonical pairing. 
\end{theorem}

This statement is significant and our primary contribution is an application of Richardson varieties to the study of $p$-adic groups, which enhances the transparency of our proof. Despite their potential, Richardson varieties have received limited attention in the context of representation theory for $p$-adic groups. In addition, we validate a criterion related to the intersection of Bruhat cells. For a precise formulation, we direct the reader to Proposition \ref{prop intersection and bruhat}.


We now briefly outline the proof of Theorem \ref{thm coincide filtrations}. The Bernstein-Casselman canonical pairing for \((\mathrm{Ind}_P^G\sigma)_V\) and \(((\mathrm{Ind}_P^G\sigma)^{\vee})_{V^-}\) is expressed in terms of the original pairing between \(\mathrm{Ind}_P^G\sigma\) and its smooth dual. A crucial component is the canonical lifting defined by the stabilization theorem. Utilizing the group-theoretic properties presented in Section \ref{s intersect bruhat cells}, we demonstrate that the canonical lift behaves well under the filtration arising from the geometric context (Proposition \ref{prop canoniclal lift geom filt}).


We now provide motivation for Theorem \ref{thm coincide filtrations}. Let $P$ be a parabolic subgroup of $G$ and let $P^-$ be its opposite parabolic subgroup. We write $P=MU$ for its Levi decomposition, where $M$ is the Levi subgroup and $U$ is the unipotent radical of $P$. Let $\sigma$ be an irreducible representation of $M$ and let $\pi$ be a quotient of $\mathrm{Ind}_P^G\sigma$. We have the following composition of maps:
\[  C_c^{\infty}(  PP^-, \sigma) \hookrightarrow \mathrm{Ind}_P^G\sigma \twoheadrightarrow \pi,
\]
where $C_c^{\infty}( PP^-, \sigma)$ is the subspace of functions in $\mathrm{Ind}_P^G\sigma$ that are compactly supported in the open subspace $P\setminus PP^-$. Taking the Jacquet functor with respect to $U^-$ yields a map $\sigma \rightarrow \pi_{U^-}$, accounting for the adjoint map in the context of second adjointness.  This map is explicitly studied in the work of Bezrukavnikov-Kazhdan \cite{BK15} from a geometric viewpoint. Theorem \ref{thm coincide filtrations} recovers this map (see Section \ref{ss adjointness map}) and can thus be viewed as a generalization of the description of the adjoint map. Our original motivation stems from questions regarding derivatives, and related discussions are presented in Section \ref{s consequences}. One interesting consequence is Corollary \ref{cor descending derivatives}.

Finally, we note that it is entirely reasonable to pose similar questions in the context of affine Hecke algebras or graded Hecke algebras. These settings are less technical and can primarily be derived from Section \ref{s weyl groups}. We will not delve into the details here. For the statement of second adjointness in these relevant contexts, the reader may refer to \cite{BC15}. Additional background on affine Hecke algebras can be found in \cite{So21}.

This paper is organized as follows. Section \ref{s weyl groups} reviews key properties of Weyl groups, including the Bruhat ordering. Section \ref{s intersect bruhat cells} examines intersections of Bruhat cells. Section \ref{s bc canonical pairing} discusses the Bernstein-Casselman pairing. Section \ref{s canoncial lift para} conducts a support computation for a canonical lift of an element in a parabolically induced module, utilizing results from Section \ref{s intersect bruhat cells}. Section \ref{s geometric lemma and dual} describes the filtrations arising from the geometric lemma. Section \ref{s match filtrations} establishes our main result on matching filtrations, as stated in Theorem \ref{thm filtrations coincide bc pairing}. Finally, Section \ref{s consequences} explores some implications.




\subsection{Acknowledgements} The author would like to thank Xuhua He and Jianghua Lu for discussions on Richardson varieties. The author would also like to thank Caihua Luo for helpful discussions, and thank Maarten Solleveld for his comments on an earlier version of this article. The author would also like to thank referees for their very helpful comments. This article is benefited from author's visit to NCTS at Taiwan in December 2024, and the author would like to thank the center for the warm hospitality. Some results in this article were announced during 1st Greater Bay Area Conference in Algebra at Shenzhen in December 2025. This project is supported in part by the National Natural Science Foundation of China (Project No. 12322120) and the Research Grants Council of the Hong Kong Special Administrative Region, China (Project No: 	17305223).

\section{Some results on Weyl groups} \label{s weyl groups}

\subsection{Basic setup and notations} \label{ss setup and notations}

Let $R$ be a (not necessarily reduced) root system and let $\Phi$ be a choice of simple roots in $R$. Let $W$ be the Weyl group associated to $R$. For a subset $S \subset \Phi$, let $W(S)$ be the subgroup of $W$ generated by the simple reflections associated to roots in $S$. For $w \in W$, the {\it $S$-part} (resp. {\it left $S$-part}), denoted by $w_S$, of $w$ is the unique minimal length element such that $ww_S^{-1} \in W(S)$ (resp. $w_S^{-1}w \in W(S)$). Let $w_o$ be the longest element in $W$. In particular, we write $w_{o,S}$ for the $S$-part of $w_o$. We let $\theta(\alpha)=-w_o(\alpha)$ for $\alpha \in R$. 

Let $w \in W$ with a reduced expression $s_r\ldots s_1$ for some simple reflections $s_1, \ldots, s_r$. A {\it subexpression} of $w$ is a sequence $\left\{  w_0, \ldots, w_r\right\}$ of elements in $W$ with the conditions:
\begin{itemize}
\item[(i)] $w_0=\mathrm{Id}$;
\item[(ii)] $w_iw_{i-1}^{-1} \in \left\{ \mathrm{Id}, s_i \right\}$ for all $1 \leq i \leq r$.
\end{itemize}
By abuse of notations, we also say $w_r$ is a subexpression of $w$. Let $\leq$ be the Bruhat ordering on $W$ i.e. $w' \leq w$ if $w'$ is a subexpression of $w$.

One basic result is the compatibility of the Bruhat ordering with the $S$-parts of Weyl group elements:

\begin{lemma} \cite[Lemmas 3.5,3.6]{De77} \label{lem bruhat order transitive S}
Let $w, w' \in W$. Let $S \subset \Phi$. If $w \leq w'$, then $w_S \leq w_S'$.
\end{lemma}


\subsection{Order reversing map} \label{ss order reversing map}

For $S, T \subset \Phi$, let $W_{S,T}$ be the set of minimal representatives of the double cosets in $W(S)\setminus W/W(T)$. We also have the following version of Lemma \ref{lem bruhat order transitive S}:

\begin{lemma} \label{lem minimal double cosets}
Let $w, w' \in W$. Let $u, u'$ be the respective minimal representatives in $W(S)wW(T)$ and $W(S)w'W(T)$. If $w \leq w'$, then $u \leq u'$.
\end{lemma}
\begin{proof}
Suppose $w\leq w'$. Let $v, v'$ be the respective minimal representatives in $wW(T)$ and $w'W(T)$. Then, by Lemma \ref{lem bruhat order transitive S}, $v \leq v'$. Note that $u, u'$ are the minimal representatives of the double cosets in $W(S)v$ and $W(S)v'$ respectively. Then, by the left version of Lemma \ref{lem bruhat order transitive S}, we have $u \leq u'$.
\end{proof}

 We shall need the following lemma:

\begin{lemma} \label{lem WT under conjugation}
Let $T \subset \Phi$. Then $w_{o,T}^{-1}W(T)w_{o,T}=W(\theta(T))$.
\end{lemma}

\begin{proof}
It is well-known that $-w_o(\Phi)=\Phi$ and so $\theta(T) \subset\Phi$. Let $x$ be the longest element in $W(T)$. Then $xw_{o,T}=w_o$. Then 
\[   w_{o,T}^{-1}W(T)w_{o,T}=w_{o,T}^{-1}x^{-1}W(T)xw_{o,T}=w_o^{-1}W(T) w_o=W(\theta(T)) .
\]
\end{proof}

We now define a map:
\begin{align} \label{eqn bijection reversing order}  \Omega_{S, T}: W_{S, T} \rightarrow W_{S, \theta(T)} 
\end{align}
such that $\Omega_{S, T}(u)$ is the minimal representative in the coset
\[   W(S) u w_{o,T}W(\theta(T)) .
\]
We may simply write $\Omega$ for $\Omega_{S,T}$ if it is clear from the context and there is no ambiguity.

\begin{proposition} \label{prop order reversion} (c.f. \cite[Lemma 3.6]{Ri92})
The map $\Omega$ is an order-reversing bijection.
\end{proposition}

\begin{proof}
We first show it is a bijection. Note that, by Lemma \ref{lem WT under conjugation},
\[  W(S)uW(T)w_{o,T}=W(S)uw_{o,T}W(\theta(T)) .
\]
Thus $\Omega(u)=\Omega(u')$ if and only if $u=u'$, which shows that $\Omega$ is a bijection.

We now show it is order-reversing. Let $u, u' \in W_{S,T}$ with $u \leq u'$. It is standard that: 
\begin{align} \label{eqn bruhat reverse longest}
    uw_o \geq u'w_o  .
\end{align}
Note that the left $\theta(T)$-part of $uw_o$ (resp. $u'w_o$) coincides with that of $uw_{o,T}$ (since $uw_o$ can be obtained from $uw_{o,T}$ by multiplying an element in $W(T)$ on the right). Thus, by Lemma \ref{lem bruhat order transitive S} and (\ref{eqn bruhat reverse longest}), $\Omega(u) \geq \Omega(u')$.
\end{proof}

\section{Intersections of Bruhat cells} \label{s intersect bruhat cells}
\subsection{Notations} \label{ss notations}
We will use the structure of linear algebraic groups as outlined in \cite{Bo91} throughout the remainder of this article. One may see \cite{CGP15} for a modern account on the subject.  Let $\mathbb G$ be a connected reductive group defined over a non-Archimedean local field $F$. Let $G=\mathbb G(F)$ be the group of $F$-rational points of $\mathbb G$. Fix a maximal split torus $T$ in $G$. Let ${}_FR=R(G,T)$ be the relative root system of $G$ with respect to $T$ (\cite[8.17]{Bo91}). Let $Z_G(T)$ be the subgroup of $G$ centralizing $T$ and let $W=N_G(T)/Z_G(T)$ be the relative Weyl group of $G$. By \cite[Theorem 21.6]{Bo91}, ${}_FR$ is a root system with the underlying Weyl group naturally isomorphic to $W$. Let $R_{nd}$ be the subset of non-divisible roots in ${}_FR$. We fix a set $\Phi$ of simple roots in ${}_FR$, which determines the set of positive roots, denoted $R^+$. Let $R^+_{nd}=R^+\cap R_{nd}$. For $S \subset \Phi$, let $R_{nd,S}$ be the root subsystem of $R_{nd} \cap (\mathbb ZS)$.

For each root $\alpha \in {}_FR$, we associate to a root subgroup $\mathbb U_{\alpha}$ defined over $F$ and write $U_{\alpha}=\mathbb U_{\alpha}(F)$ \cite[Proposition 21.9]{Bo91}. For $\alpha \in R_{nd}$, let $U_{(\alpha)}=U_{\alpha}$ if $2\alpha \notin R$ and let $U_{(\alpha)}$ be the subgroup generated by $U_{\alpha}$ and $U_{2\alpha}$ otherwise. The group $U$ is directly spanned by $U_{(\alpha)}$ ($\alpha \in R_{nd}^+$) taken in any order i.e. 
\[  \prod_{\alpha \in R_{nd}^+} U_{(\alpha)} \rightarrow U
\]
is an isomorphism (in any order taken) \cite[Proposition 21.9]{Bo91}. 

Let $B$ be the minimal parabolic $F$-subgroup determined by $T$ and $\Phi$. A {\it standard parabolic subgroup} of $G$ is a subgroup generated by $Z_G(T)$, and $U_{(\alpha)}$ for $\alpha \in R_{nd}^+\cup R_{nd, S}$ for some subset $S \subset \Phi$ (\cite[Section 21.11]{Bo91}). For a standard parabolic subgroup $P$ of $G$, we also write $M_P$ to be the Levi subgroup in $P$, write $U_P$ to be the unipotent subgroup of $P$, and write $U_P^-$ to be the unipotent subgroup of the parabolic subgroup opposite to $P$.

\subsection{Open compact subgroups} \label{ss open compact subgp}

It is well-known that $G$ is an $l$-group i.e. Hausdorff topological group such that the identity has a basis of neighborhoods which are open and compact (see \cite[Ch II, Section 2.1]{Be92}). For an open compact subgroup $K$ of $G$ and $\alpha \in R_{nd}$, let $K_{(\alpha)}=K\cap U_{(\alpha)}$ and $K'=K\cap T$. An open compact subgroup $K$ of $G$ is said to be {\it in a good position} if the natural map 
\[ K'\times \prod_{\alpha \in R_{nd}} K_{(\alpha)} \rightarrow K
\]
is a bijection, for any ordering on $\alpha$'s. 

In the context below, let $K_0$ be a special open compact subgroup in $G$. All chosen open compact subgroups are assumed to be in $K_0$ and are normal in $K_0$. We can always choose such sufficiently small open compact subgroup which is also in a good position (see Bruhat's theorem \cite[Page 41]{Be92}, also see \cite{BD84, BH03}). We shall use this several times. In particular, this also implies one can always find a sufficiently small open compact subgroup $K$ with respect to any parabolic subgroup $P$ satisfying: 
\[   K =(K\cap U_P)\cdot (K \cap M_P) \cdot (K\cap U_P^-) .
\]

\subsection{ Bruhat cells}

We keep using above notations. We have fixed a set of simple roots and we shall also use the notions in Section \ref{ss setup and notations} such as the Bruhat ordering $\leq$. For each $w \in W$, we choose a representative $\dot{w}$ of $w$ in $N_G(T)$. According to \cite[Theorem 21.26]{Bo91}, the Zariski closure of $B\dot{w}B$
\[ \overline{B\dot{w}B} = \bigsqcup_{w' \leq w} B\dot{w}'B ,
\]
which is also closed in $G$ under the topology of an $l$-group.

\subsection{Intersection of Bruhat cells}



The intersection $B\dot{w}B \cap B\dot{w}'B^-$ typically refers to as a Richardson variety when the underlying field is $\mathbb C$. This intersection is significant in the study of algebraic groups and has been explored in depth in works such as \cite{De86}. A refinement of this study is presented in \cite{Ri92}. In this context, we will outline the intersection and introduce some essential combinatorial elements that are critical to understanding its structure.

\begin{definition}
Let $w \in W$ with a reduced expression $s_r\ldots s_1$ for simple reflections $s_1, \ldots, s_r$ in $W$. A subexpression $\left\{ w_0, \ldots, w_{r} \right\}$ of $w$ is said to be {\it distinguished} if $w_0=\mathrm{Id}$, and for each $1 \leq i \leq r$, 
\begin{align} \label{eqn distinguished subexp cond}
w_{i} \leq s_iw_{i-1}.
\end{align}

Note that in the condition of a distinguished subexpression, if the inequality (\ref{eqn distinguished subexp cond}) is strict, then $w_{i}=w_{i-1}$. 
\end{definition}

\begin{theorem}\cite[Remark 2.4, Proposition 5.3]{De86} \label{thm subexpress = distinguished}
The set of distinguished subexpressions of $w$ coincides with the set of subexpressions.
\end{theorem}

\begin{proposition} (c.f. \cite{De86}) \label{prop intersection and bruhat}
Let $w, w' \in W$. Then $B\dot{w}B \cap B\dot{w}'B^- \neq \emptyset$ if and only if $w'\leq w$.
\end{proposition}

\begin{proof}
Since \cite[Theorem 1.1]{De86} is stated for groups over $\mathbb C$, we shall give some details following his proof in our case. We make some simplifications and hence we do not have (and do not need) a concrete description on the intersection. 

We first prove the only if direction. Suppose $B\dot{w}B \cap B\dot{w}'B^-\neq \emptyset$. Write the reduced expression for $w$ as $s_r \ldots s_1$. Then we have:
\[  \dot{s}_r\ldots \dot{s}_1u \in B\dot{w}'B^- .
\]
for some $u \in U$. Let $z_i' \in W$ be the unique element such that $\dot{s}_i\ldots \dot{s}_1u \in B\dot{z}_i'B^-$. Now, one begins with $u$ and inductively applies $\dot{s}_i$ on the left. Then 
\[   \dot{s}_i\ldots \dot{s}_1u \in \dot{s}_iB\dot{z}_{i-1}'B^- .
\]
By using $\dot{s}_iB\dot{z}'_{i-1}B^- \subset B\dot{s}_i\dot{z}_{i-1}'B^- \cup B\dot{z}_{i-1}'B^-$, we have that $z_i' =s_iz_{i-1}'$ or $z_{i-1}'$. Note that $u \in B\dot{z}_0'B^-$ and so $z_0'=\mathrm{Id}$. In other words, by the definition of the Bruhat ordering, $z_r'$ is a subexpression of $w$. This proves the only if direction.

We now consider the if direction. Let $z_0=\mathrm{Id}, z_1, \ldots, z_r$ be a sequence of elements in $W$ representing a subexpression $w'$ of $w$. By Theorem \ref{thm subexpress = distinguished}, we shall assume $z_1,\ldots, z_r$ form a distinguished subexpression. We now construct inductively elements in the intersection:
\[  B\dot{s}_{q}\ldots \dot{s}_1B \cap B\dot{z}_{q}B^- .
\]
When $q=r$, we shall have the desired element in the intersection. 

When $q=0$, it is clear that $B \cap BB^-\neq \emptyset$ and we can pick a particular element in the intersection. Suppose we have constructed an element in
\begin{align} \label{eqn intersect bruhat cells}
  B\dot{s}_{q}\ldots \dot{s}_1B \cap B\dot{z}_{q}B^- .
\end{align}
We now consider 
\[  B\dot{s}_{q+1}\ldots \dot{s}_1B \cap B\dot{z}_{q+1}B^- 
\]
and consider the following cases:
\begin{itemize}
\item Case 1: $z_{q+1} \neq z_q$ i.e. $z_{q+1}=s_{q+1}z_q$. By multiplying an element in $B$ on the left in (\ref{eqn intersect bruhat cells}), we may assume the intersection $ B\dot{s}_{q}\ldots \dot{s}_1B \cap B\dot{z}_{q}B^- $ contains an element in $B\dot{z}_qB^-$ of the form:
\[    \dot{z}_qb' 
\]
for some $b' \in B^-$. Then $\dot{s}_{q+1}\dot{z}_qb' \in B\dot{z}_{q+1}B^-$ and is also in $\dot{s}_{q+1}B\dot{s}_q\ldots \dot{s}_1B \subset B\dot{s}_{q+1}\dot{s}_q\ldots \dot{s}_1B$ (\cite[Proposition 21.22]{Bo91}), as desired.
\item Case 2: $z_{q+1}=z_q$. Let $\alpha_{q+1}$ be the simple root associated to the simple reflection $s_{q+1}$. Let $v \in U_{(-\alpha_{q+1})}=\dot{s}_{q+1}U_{(\alpha_{q+1})}\dot{s}_{q+1}^{-1} $ be a non-identity element. Note that  $v \in \dot{s}_{q+1}B\dot{s}_{q+1}^{-1}\subset B\cup B\dot{s}_{q+1}B$ by \cite[Theorem 21.15]{Bo91}. Since $v \notin B$, $v$ can be written in the form 
\[   x' \dot{s}_{q+1}x 
\]
for some $x \in U_{\alpha_{q+1}}$ and $x' \in B$, and so $v\dot{z}_qb'$ is of the form:
\[    x'\dot{s}_{q+1}x \dot{z}_q b',
\]
where $b' \in B^-$. Then, we have:
\[   \dot{s}_{q+1} x \dot{z}_q b'  \in \dot{s}_{q+1}B\dot{s}_q\ldots \dot{s}_1B \subset B\dot{s}_{q+1}\ldots \dot{s}_1B.
\]
Hence, $v\dot{z}_qb' \in B\dot{s}_{q+1}\dots \dot{s}_1B$. 

On the other hand, the distinguishedness above implies that $z_q^{-1}(\alpha_{q+1}) >0$, and so $\dot{z}_q^{-1}v\dot{z}_q \in U_{\alpha}^- \subset B^-$. Hence, the element $v \dot{z}_qb'$ is also in $B\dot{z}_{q+1}B^-=B\dot{z}_qB^-$. This proves this case.
\end{itemize}
\end{proof}


\subsection{Consequences on other intersections} \label{ss consequences on intersections}

Let $P, Q$ be standard parabolic subgroups of $G$. Let $S$ and $T$ be the respective subsets of $\Phi$ determining $P$ and $Q$. We use the notion $W_{S,T}$ as the one in Section \ref{ss order reversing map}. Then 
\[ G= \bigsqcup_{w \in W_{S,T}} P\dot{w}Q .
\]
Indeed, we can also use $W_{S,T}$ to parametrize the double cosets of $P\setminus G/Q^-$:

\begin{lemma} \label{lem decompose by opposite}
We use the notations above. Let $w, w' \in W_{S,T}$. Then 
\[  P\dot{w}Q^-=P\dot{w}'Q^-   \quad \Longleftrightarrow \quad w =w'  .
\]
Moreover, 
\[   G=\bigsqcup_{w \in W_{S,T}} P\dot{w}Q^- .
\]
\end{lemma}

\begin{proof}
Let $\bar{Q}=\dot{w}_{o,T}^{-1}(Q^-)\dot{w}_{o,T}$. Since $w_{o,T}^{-1}$ sends all negative roots outside the root subsystem generated by $T$ to positive roots, $\bar{Q}$ is still a standard parabolic subgroup of $G$. Moreover, the subset of $\Phi$ determining $\bar{Q}$ is $\theta(T)$. Then, we have:
\[  P\dot{w}Q^-= P\dot{w}\dot{w}_{o,T}\bar{Q}\dot{w}_{o,T}^{-1}=P\dot{v}\bar{Q}\dot{w}_{o,T}^{-1} ,
\]
where $v=\Omega_{S,T}(w)$ ($\Omega_{S,T}$ is defined in (\ref{eqn bijection reversing order})). Then the two assertions follow from Proposition \ref{prop order reversion}. 
\end{proof}

\begin{lemma} \label{lem bruhat decompose for parabolic}
 Let $w \in W_{S,T}$. Then 
\[  P\dot{w}Q = \bigsqcup_{x\in W(S), y \in W(T)}B\dot{x}\dot{w}\dot{y}B ,\quad   P\dot{w}Q^- =\bigsqcup_{x \in W(S), y \in W(T)} B\dot{x}\dot{w}\dot{y}B^-.
\]
\end{lemma}

\begin{proof}
We will only prove the second equation, as the first is simpler. Let $\bar{Q}=\dot{w}_oQ^-\dot{w}_o^{-1}$. We remark that $\bar{Q}$ coincides with the one defined in the proof of Lemma \ref{lem decompose by opposite} and so is still a standard parabolic subgroup as discussed in the proof of that lemma. The second assertion follows from:
\begin{align*}
 P\dot{w}Q^- &= P\dot{w}\dot{w}_o\bar{Q}\dot{w}_o^{-1} \\
             &= \bigsqcup_{x\in W(S), y \in W(\theta(T))}B\dot{x}B\dot{w}\dot{w_o}B\dot{y}B\dot{w}_o^{-1} \\
						 &= \bigsqcup_{x\in W(S), y \in W(\theta(T))}B\dot{x}B\dot{w}\dot{w_o}\dot{y}B\dot{w}_o^{-1} \\
						 &= \bigsqcup_{x \in W(S), y \in W(\theta(T))} B\dot{x}\dot{w}\dot{w}_o\dot{y}B\dot{w}_o^{-1} \\
						 &= \bigsqcup_{x \in W(S), y \in W(\theta(T))} B\dot{x}\dot{w}\dot{w}_o\dot{y}\dot{w}_o^{-1}B^- \\
						 &= \bigsqcup_{x \in W(S), z \in W(T)} B\dot{x}\dot{w}\dot{z}B^-,
\end{align*}
where the third equality follows from that for $v \in W$ and a simple reflection $s$ in $W$,
\[  B\dot{v}\dot{s}B\subset   B\dot{v}B\dot{s}B \subset B\dot{v}\dot{s}B \cup B\dot{v} B
\]
and any subexpression of an element of $W(\theta(T))$ is still in $W(\theta(T))$, and the fourth equality follows from $\dot{w}_oB^-\dot{w}_o^{-1}=B$, and the fifth equality follows from $w_oW(\theta(T))w_o^{-1}=W(T)$.
\end{proof}

\begin{corollary} \label{cor parabolic intersection}
We use the notations above. Let $w, w' \in W_{S,T}$. Then,  
\[  P\dot{w}Q\cap P \dot{w}'Q^- \neq \emptyset ,
\]
if and only if $w' \leq w$.
\end{corollary}

\begin{proof}
We first prove the if direction. Suppose $w \leq w'$. Then $B\dot{w}B \cap B\dot{w}'B^-\neq \emptyset$ by Proposition \ref{prop intersection and bruhat}
 and so $P\dot{w}Q \cap P\dot{w}'Q^-\neq \emptyset$.

We now prove the only if direction. Suppose $P\dot{w}Q \cap P\dot{w}'Q^-\neq \emptyset$. Then, by  $B\dot{x}\dot{w}\dot{y}B \cap B\dot{x}'\dot{w}'\dot{y}'B^-\neq \emptyset$ for some $x,x' \in W(S)$ and $y,y' \in W(T)$. Hence, $x'w'y' \leq xwy$ by Proposition \ref{prop intersection and bruhat}. By Lemma \ref{lem minimal double cosets}, we have $w' \leq w$.
\end{proof}

\subsection{Closure of Bruhat cells}

\begin{lemma} \label{lem algebraic closure of bwb}
We use the notations in the beginning of Section \ref{ss consequences on intersections}. Let $w \in W_{S, T}$. Then
\[    \overline{P\dot{w}Q} = \bigsqcup_{v\in W_{S,T}, v\leq w} P\dot{v}Q .
\]
\end{lemma}

\begin{proof}
\cite[Theorem 21.26]{Bo91} is primarily stated for $P$ and $Q$ to be the minimal standard parabolic subgroup. The general form is a direct consequence of Lemmas \ref{lem minimal double cosets} and \ref{lem bruhat decompose for parabolic}, and the case of a minimal standard parabolic subgroup in \cite[Theorem 21.26]{Bo91}.
\end{proof}

\begin{lemma} \label{lem algebraic closure of bwb-}
We again use the notations in the beginning of Section \ref{ss consequences on intersections}. Let $w \in W_{S, T}$. Then
\[   \overline{P\dot{w}Q^-} =\bigsqcup_{v \in W_{S,T}, w\leq v} P\dot{v}Q^- .
\]
\end{lemma}

\begin{proof}
We define $\bar{Q}$ as in the proof of the previous two lemmas. Note that
\begin{align*}
 P\dot{w}Q^- &= P\dot{w}\dot{w}_{o,T}\bar{Q}\dot{w}_{o,T}^{-1} .
\end{align*}
Now,
\begin{align*}
  \overline{P\dot{w}Q^-} & = \overline{P\dot{w}\dot{w}_{o,T}\bar{Q}}\dot{w}_{o,T}^{-1} \\
	                       & = \bigsqcup_{v' \in W_{S, \theta(T)}: v' \leq \Omega_{S,T}(w)} P\dot{v}'\bar{Q} \dot{w}_{o,T}^{-1} \\
												 & = \bigsqcup_{v \in W_{S, T}: \Omega_{S,T}(v) \leq \Omega_{S,T}(w)} P\dot{v} Q^- \\
												 & = \bigsqcup_{v \in W_{S, T}: w\geq v } P\dot{v} Q^- ,
\end{align*}
where the first equality follows from that a group action is continuous; the second equality follows from Lemma \ref{lem algebraic closure of bwb}; the third equality follows from the definition of $\Omega_{S,T}$ and Corollary \ref{cor parabolic intersection}; the forth equality follows from Proposition \ref{prop order reversion}.
\end{proof}

\section{Bernstein-Casselman canonical pairing} \label{s bc canonical pairing}

\subsection{Stable operator}

For a vector space $V$ with an operator $\omega \in \mathrm{End}(V)$, we say that $\omega$ is {\it stable} if $\omega$ is invertible on $\mathrm{im}~\omega$ and 
\[ V=\mathrm{ker}~\omega \oplus \mathrm{im}~\omega.  \] 

\begin{proposition} \label{prop stable for dual}
Let $\pi$ be a smooth representation of $G$ and let $\omega \in \mathrm{End}(\pi)$ be stable. Let $\pi^{\vee}$ be the smooth dual of $\pi$. Define $\omega^{\vee}$ to be the operator on $\pi^{\vee}$ given by $(\omega^{\vee}.f)(v)=f(\omega.v)$, where $f \in \pi^{\vee}$ and $v \in \pi$. Then $\omega^{\vee}$ is also stable.
\end{proposition}

\begin{proof}
Let $X$ be the subspace of functions in $\pi^{\vee}$ vanishing on $\mathrm{ker}~\omega$. We claim that $\mathrm{im}~\omega^{\vee}=X$. 

Let $f \in \mathrm{im}~\omega^{\vee}$. Then $f=\omega^{\vee}.\widetilde{f}$ for some $\widetilde{f} \in \pi^{\vee}$. For $x \in \mathrm{ker}(\omega)$, $f(x)=(\omega^{\vee}.\widetilde{f})(x)=f(\omega.x)=0$. Hence, $f \in X$. This proves $\mathrm{im}~\omega^{\vee} \subset X$. 

Let $f \in X$. For each $x \in \mathrm{im}~\omega$, we can write $x= \omega^{-1}.\widetilde{x}$ for some $\widetilde{x} \in \mathrm{im}~\omega$. Then, define $\widetilde{f}(x)=f(\widetilde{x})$ if $x \in \mathrm{im}~\omega$ and $\widetilde{f}(x)=0$ if $x \in \mathrm{ker}~\omega$. This then implies that $(\omega^{\vee}.\widetilde{f})(x)=f(x)$ for all $x$. Since $f$ is smooth, $\widetilde{f}$ is also smooth and in $\pi^{\vee}$. This shows that $X \subset \mathrm{im}~\omega^{\vee}$. 

Let $Y$ be the subspace of functions in $\pi^{\vee}$ vanishing on $\mathrm{im}~\omega$. Then one applies a similar above argument to show that $Y=\mathrm{ker}~\omega^{\vee}$.
\end{proof}

\subsection{Jacquet functor} \label{s jacquet functor}

Let $\pi$ be a smooth representation of $G$. Let $P$ be a standard parabolic subgroup of $G$. Define the Jacquet module of $\pi$ with respect to $P$ to be:
\[  \pi_{U_P} := \delta^{-1/2} \cdot \frac{\pi}{\langle u.x-x : x \in \pi, u \in U_P \rangle} ,
\]
where $\delta$ is the modular character of $P$. For the unipotent radical $U_P^-$ in the parabolic subgroup $P^-$ opposite to $P$, we can define similarly for $\pi_{U_P^-}$ by replacing $U_P$ with $U_P^-$ in the definition and using the modular character of $P^-$. There is a natural $M_P$-action on $\pi_{U_P}$ and $\pi_{U_P^-}$.

\subsection{Stabilization Theorem}

For a smooth representation $\pi$ of $G$ and an open compact subgroup $K$ of $G$, we define
\[  \pi_K(g).v = \int_{G} \mathrm{ch}_{KgK}(h) \cdot \pi(h)v ~dh =\int_{KgK}\pi(h)v~dh, 
\]
where $\mathrm{ch}_{KgK}$ is the characteristic function of $KgK$, and $dh$ is the normalized Haar measure on $G$ with $\int_{K_0} dh=1$. 

Let $P$ be a standard parabolic subgroup of $G$. Let $Z(M_P)$ be the center of $M_P$. An element $a \in Z(M_P)$ is said to be {\it strictly dominant} for $M_P$ if for any two open compact subgroups $H_1, H_2$ in $U_P$, there exists an integer $n \geq 0$ such that $a^nH_1a^{-n} \subset H_2$; and furthermore, for any two open compact subgroups $H_1', H_2'$ in $U_P^-$, there exists an integer $n'\geq 0$ such that $a^{-n'}H_1'a^{n'} \subset H_2'$. The existence of such element is shown in \cite[Lemma 6.14]{BK98} and allows one to find an 'inverse' of the projection $\pi^K \rightarrow \pi_U^{K \cap M_P}$ as follows.

\begin{theorem}[Jacquet's lemma, Bernstein's stablization theorem] \cite{BD84, Be92}\label{thm jacq lem stab}
Let $\pi$ be a smooth representation of $G$. Let $P$ be a standard parabolic subgroup of $G$, let $U=U_P$. Let $K$ be an open compact subgroup of $G$ in a good position. Let $K_M=M_P\cap K$. Let $a$ be a strictly dominant element for $M_P$. Then,
\begin{enumerate}
\item \cite[Proposition IV.6.1]{Re10} There exists an integer $i_0\geq 0$ such that for all $i \geq i_0$, $\pi_K(a^i)$ is stable. Moreover, the natural projection $\mathrm{pr}$ from $\pi^K$ to $\pi_U^{K_M}$ induces a bijection from $\mathrm{im}~\pi_K(a^i)$ to $\pi_U^{K_M}$ for all $i \geq i_0$. 
\item \cite[Th\'eor\'eme IV.6.1]{Re10} The natural projection $\mathrm{pr}: \pi^K \rightarrow \pi_U^{K_M}$ has a right inverse, as a linear map.
\item \cite[Pages 261, 262]{Re10} For $i \geq i_0$, the projection $\mathrm{pr}: \pi^K \rightarrow \pi_U^{K_M}$ sends $\mathrm{ker}~\pi_K(a^i)$ to zero.
\end{enumerate}
\end{theorem}

\noindent
{\it Sketch of some proofs.}
Here we provide some explanations on some parts of the theorem for the convenience of the readers. We assume the stabilization part: there exists an integer $i_0 \geq 0$ such that for all $i \geq i_0$, $\pi_K(a^i)$ is stable. For the remaining part of (1), for $x \in \pi_U^{K_M}$, a classical result of Jacquet (usually referred to as Jacquet's lemma) allows one to find a large $i_x$ such that $a^{i_x}.x$ is in the image $\mathrm{im}~\mathrm{pr}$. The stabilization theorem says that one can find one $i^*$ such that $a^{i*}$ maps from $\pi_U^{K_M}$ to $\mathrm{im}~\mathrm{pr}$. Since $a^{i^*}$ is invertible on $\pi_U^{K_M}$, we then have that  $\mathrm{im}~\mathrm{pr}$ is equal to the entire $\pi_U^{K_M}$.

Note that (2) follows from (1). We now consider (3). Suppose $x \in \mathrm{ker} \pi_K(a^i)$. This implies that $\int_{a^{-i}K_Ua^i} u.x~ du=0$. Another classical result of Jacquet then implies that $x$ is in the kernel of the projection. This shows (3). \qed

\begin{corollary} \label{cor canonical lift}
We use the notations in Theorem \ref{thm jacq lem stab}. Let $U^-=U_P^-$. Then the followings also hold:
\begin{enumerate}
\item With the same $i_0$ in Theorem \ref{thm jacq lem stab}, for $i \geq i_0$, $\pi^{\vee}_K(a^{-i})$ is stable. The natural projection from $(\pi^{\vee})^K$ to $(\pi^{\vee})_{U^-}^{K_M}$ induces a bijection from $\mathrm{im}~\pi^{\vee}_K(a^{-i})$ to $(\pi^{\vee})_{U^-}^{K_M}$ for all $i \geq i_0$. 
\item The natural projection $\mathrm{pr}: (\pi^{\vee})^K \rightarrow (\pi^{\vee})_{U^-}^{K_M}$ has a right inverse, as a linear map.
\item For $i \geq i_0$, the natural projection $\mathrm{pr}: (\pi^{\vee})^K \rightarrow (\pi^{\vee})_{U^-}^{K_M}$ sends $\mathrm{ker}~\pi^{\vee}_K(a^{-i})$ to zero.
\end{enumerate}
\end{corollary}

\begin{proof}
The stability part of (1) follows from Proposition \ref{prop stable for dual}. Then the remaining part follows from the stability part and classical treatments of Jacquet (see the sketch of the proof for Theorem \ref{thm jacq lem stab} above). 
\end{proof}

\begin{definition}
Let $\pi$ be a smooth representation of $G$. Let $P$ be a standard parabolic subgroup of $G$. Let $x \in \pi_{U_P}$ (resp. $y \in \pi_{U_P^-}$). We shall say that an element $\widetilde{x}$ (resp. $\widetilde{y}$) in $\pi$ is a {\it lift} of $x$ (resp. $y$) if the image of $\widetilde{x}$ (resp. $\widetilde{y}$) under the natural projection map $\pi \rightarrow \pi_{U_P}$ (resp. $\pi \rightarrow \pi_{U_P^-}$) is $x$ (resp. $y$). 
\end{definition}

We now define the notion of canonical lifts with respect to a fixed choice of an open compact subgroup $K$.

\begin{definition} \cite{Ca93, Be92}
Let $\pi$ be a smooth representation of $G$. Let $P$ be a standard parabolic subgroup of $G$. Let $K$ be an open compact subgroup of $G$. Let $K_M=K \cap M_P$. Let $a$ be a strictly dominant element for $M_P$. Let $i$ be a sufficiently large integer such that $\pi_K(a^i)$ and so $\pi^{\vee}_K(a^{-i})$ are stable.
\begin{enumerate}
\item Let $x \in \pi_{U_P}^{K_M}$.  The unique element $\widetilde{x} \in \mathrm{im}~\pi_K(a^i)$ mapped to $x$ in Theorem \ref{thm jacq lem stab}(1) is called the {\it canonical lift} of $x$ (with respect to $K$). 
\item Let $y \in (\pi^{\vee})_{U_P^-}^{K_M}$. The unique element $\widetilde{y} \in \mathrm{im}~\pi^{\vee}_K(a^{-i})$ mapped to $y$ in Corollary \ref{cor canonical lift}(1) is called the {\it canonical lift} of $y$ (with respect to $K$).  
\end{enumerate}
\end{definition}

\begin{example}
Let $G$ be a split reductive $p$-adic group and let $I$ be the Iwahori subgroup of $G$. Let $B$ be the Borel subgroup of $G$, let $T$ be the maximal torus and let $U$ be the unipotent radical of $B$. Let $\pi$ be a smooth representation in an Iwahori component of $G$. It is a classical result that the projection $\mathrm{pr}$ from $\pi$ to $\pi_U$ induces an isomorphism from $\pi^I$ to $\pi_U^{I\cap T}$. (For some details, see e.g. \cite[Theorem 6.1]{BM93} for a description of parabolically induced modules in terms of Hecke algebra actions and then apply an adjointness to obtain a description of the Jacquet functor.) In other words, any element in $\pi^I$ is the canonical lift of its image under $\mathrm{pr}$.

\end{example}

\subsection{Bernstein-Casselman canonical pairing}

We now state the canonical pairing, which is due to Casselman \cite{Ca93} for admissible representations, following results of Jacquet. The stabilization theorem of Bernstein above removes the admissibility condition. 

\begin{theorem}[Bernstein, Casselman] (see \cite[Proposition VI 9.6]{Re10}) \label{thm bernstein-casselman}
Let $\pi$ be a smooth representation of $G$. Let $P$ be a standard parabolic subgroup of $G$. Let $M=M_P$, let $U=U_P$ and let $U^-=U_P^-$. Then the followings hold:
\begin{itemize}
\item[(1)] There is a functorial isomorphism $(\pi_U)^{\vee} \cong (\pi^{\vee})_{U^-}$, as $M$-representations. 
\item[(2)] Moreover, the isomorphism in (1) is explicitly given as follows. For $x \in \pi_U$ and $f \in (\pi^{\vee})_{U^-}$, let $K$ be an open compact subgroup of $G$ in a good position such that  $x \in \pi_U^{K \cap M}$ and $f \in (\pi^{\vee})_{U^-}^{K\cap M}$. Let $\widetilde{x} \in \pi^K$ and $\widetilde{f} \in (\pi^{\vee})^K$ be the canonical lifts of $x $ and $f $ (with respect to $K$). Then the isomorphism is given by the pairing $\langle x, f \rangle_M:=\widetilde{f}(\widetilde{x})=\langle \widetilde{x}, \widetilde{f}\rangle$, which is independent of the choice of $K$. 
\end{itemize}
\end{theorem}


\section{Canonical lifts of some parabolically induced modules} \label{s canoncial lift para}

\subsection{Parabolic induction}

Let $P$ be a standard parabolic subgroup of $G$. Let $\sigma$ be a smooth representation of $M_P$, inflated to a $P$-representation. Define
\[  \mathrm{Ind}_P^G\sigma := \left\{ f: G \rightarrow \sigma | f \mbox{ is smooth},\ f(pg)=\delta(p)^{1/2} p.f(g) \mbox{ for $p \in P$} \right\} ,
\]
where $\delta$ is the modular character of $P$. We shall regard $\mathrm{Ind}_P^G\sigma$ as a $G$-representation via the right translation i.e. for $f \in \mathrm{Ind}_P^G\sigma$,
\[  (g.f)(h)=f(hg)
\]
for any $g, h \in G$.

For $f \in \mathrm{Ind}_P^G\sigma$, define
\[  \mathrm{supp}(f)=\left\{  x \in G: f(x) \neq 0 \right\} ,
\]
called the {\it support} of $f$.

\subsection{Support of a function under the action of a strictly dominant element}

We will begin by presenting two lemmas about group structure, continuing from Section \ref{s intersect bruhat cells}.

\begin{lemma} \label{lem subset on Cartan decomposition}
Let $K$ be an open compact subgroup of $G$ in a good position and let $Q$ be a standard parabolic subgroup of $G$. Let $M=M_Q$, $U=U_Q$ and $U^-=U_Q^-$. Let $a$ be a strictly dominant element for $M$. Then, for $i \geq 0$, $Ka^iK \subset QU^-$.
\end{lemma}

\begin{proof}
Since $K$ is in a good position, we have the Iwahori decomposition of $K$:
\[   K= K_U \cdot K_M \cdot K_{U^-} ,
\]
where $K_*=*\cap K$ for $*=M, U, U^-$. 

Let $g \in Ka^iK$. Write $g=k_1a^ik_2$ for $k_1, k_2 \in K$. By the Iwahori decomposition, write $k_2=umu'$ for $u \in K_U$, $m \in K_M$ and $u' \in K_{U^-}$. Then 
\[  g=k_1a^ik_2=k_1a^i(umu')=k_1(a^iua^{-i})a^imu' \in K  Q^-.
\]
which follows from $a^iua^{-i}\in K$ (as $a$ is straightly dominant and $i \geq 0$) and $a^imu' \in Q^-$. 

On the other hand, by the Iwahori decomposition again, $K \subset K_U\cdot K_M \cdot K_{U^-}\subset Q  Q^-$. Thus, we also have $g \in Q Q^-=QU^-$. This shows the first inclusion.

\end{proof}

\begin{lemma} \label{lem support under a action}
We use the notations in Lemma \ref{lem subset on Cartan decomposition}. Let $P$ be another standard parabolic subgroup of $G$. Let $S$ and $T$ be the subsets of $\Phi$ determining $P$ and $Q$ respectively. Let $u, w \in W_{S,T}$. Let $g \in G$ and let $m \in M$. Suppose $gka^{-i} \in P\dot{w}mK$ for some $i \geq 0$, for some $k \in K$ and for some $m \in M_Q$. If $u \not\leq w$, then 
\[  PgK \cap P\dot{u}Q^- =\emptyset ,
\]
and in particular, $g \notin P\dot{u}Q^-$.
\end{lemma}

\begin{proof}
The hypothesis gives that $g \in P\dot{w}mKa^{i}K$. Thus 
\[ PgK \subset  P\dot{w}mKa^iK \]
Now, by Lemma \ref{lem subset on Cartan decomposition}, we then have:
\[ PgK  \subset P\dot{w}QU_Q^- .\]

On the other hand, since $U_Q^- \subset Q^-$, we also have:
\[  P\dot{w}QU_Q^- \cap P\dot{u}Q^- \neq \emptyset \quad \Longleftrightarrow \quad P\dot{w}Q\cap P\dot{u}Q^- \neq \emptyset . \]
Now the lemma follows from Corollary \ref{cor parabolic intersection}.
\end{proof}

As seen in Theorem \ref{thm bernstein-casselman}, computing the canonical pairing may involve computing the action of $\pi_K(a^i)$ or $\pi_K(a^{-i})$ for some strictly dominant $a$ for some Levi subgroup of $G$. We shall need the following related proposition later.

\begin{proposition} \label{lem support after acting by ak}
Let $K$ be an open compact subgroup of $G$ in a good position. Let $P, Q$ be standard parabolic subgroups of $G$. Let $S, T$ be the subsets of $\Phi$ determining $P$ and $Q$ respectively. Let $\sigma$ be a smooth representation of $M_P$ and let $\pi=\mathrm{Ind}_P^G\sigma$. Let $w \in W_{S,T}$. Let
\[ \mathcal O^-_w= \bigsqcup_{w' \in W_{S,T}: w' \leq w}  P\dot{w}'Q^-. \]
 Let $f$ be a function in $\mathrm{Ind}_P^G\sigma$ such that $\mathrm{supp}(f)\subset P\dot{w}mK$ for some $m \in M_Q$. For any strictly dominant element $a$ for $M_Q$ and for $i \geq 0$, $\pi_K(a^{-i})f$ has support in $\mathcal O_w^-$.
\end{proposition}

\begin{proof}
Let $\widetilde{f}=\pi_K(a^{-i})f$. Let $g \in G$ such that $\widetilde{f}(g)\neq 0$. By definition,
\[ \widetilde{f}(g)= \int_{Ka^{-i}K} f(gh)~dh 
\]
and so $gka^{-i}  \in P\dot{w}mK$ for some $k \in K$. By Lemma \ref{lem support under a action}, $g \in P\dot{w}'Q^-$ for some $w' \leq w$. This implies the proposition. 
\end{proof}

\section{Geometric lemma and its dual} \label{s geometric lemma and dual}

\subsection{Geometric lemma} \label{ss geometric lemma}

We use the notations as in Section \ref{ss notations}. Let $P$ and $Q$ be standard parabolic subgroups of $G$. Let $S$ and $T$ be the subsets of $\Phi$ determining $P$ and $Q$ respectively. Then, according to \cite[Section 21.16]{Bo91}, there is a one-to-one correspondence between
\[   W(S) \setminus W/W(T) \longleftrightarrow  P\setminus G/Q .
\]
Recall that $W_{S,T}$ is the set of minimal representatives in $W(S) \setminus W/W(T)$. Let $r=|W_{S,T}|$. We enumerate the elements $u_1, \ldots, u_r$ in $W_{S,T}$ such that $i<j$ implies $l
(u_i)\leq l(u_j)$. This, in particular, gives that $u_i\leq u_j$ implies $i \leq j$. 

Let 
\[ \mathcal O_i=\bigsqcup_{i'\geq i} P\dot{u}_{i'}Q. \]
 Then, by \cite[Theorem 21.26]{Bo91} (also see Lemma \ref{lem algebraic closure of bwb}), $\mathcal O_i$ is open in $G$ in the Zariski topology and so is open in $G$ as an $l$-group. Let $\mathcal I_i:=\mathcal I_i(Q, \sigma)$ be the subspace of $\mathrm{Ind}_P^{G}\sigma$ with functions supported in $\mathcal O_i$. By \cite[Proposition 1.8]{BZ76}, we have:
\[   \mathcal I_i/\mathcal I_{i-1} \cong C_c^{\infty}(P\dot{u}_iQ, \sigma) ,
\]
where $C_c^{\infty}(P\dot{u}_iQ, \sigma)$ is the space of smooth $\sigma$-valued functions on $P\dot{u}_iQ$ satisfying $f(pg)=p.f(g)$ with the support to be compactly supported modulo $P$.

Let $\mathcal J_i:=\mathcal J_i(Q, \sigma)$ be the image of the projection from $\mathcal I_i$ to $(\mathrm{Ind}_P^{G}\sigma)_{U_Q}$. This then induces a filtration on $(\mathrm{Ind}_P^{G}\sigma)_{U_Q}$:
\[    0\subseteq \mathcal J_r  \subseteq \ldots \subseteq \mathcal J_1=(\mathrm{Ind}_P^{G}\sigma)_{U_Q} .
\]

As shown in \cite{BZ77}, there is an isomorphism between
\begin{align} \label{eqn isomoprhism geometric lemma}
  C^{\infty}_c(P\dot{w}Q, \sigma)_{U_Q} \cong \mathrm{Ind}^{M_Q}_{P^w\cap M_Q}(\sigma_{P\cap \dot{w}U_Q\dot{w}^{-1}})^w , 
\end{align}
where $P^w=\dot{w}^{-1}P\dot{w}$ and $(\sigma_{P\cap \dot{w}U_Q\dot{w}^{-1}})^w$ is a twist sending from the $P\cap \dot{w}M_Q\dot{w}^{-1}$-representation $\sigma_{P\cap \dot{w}U_Q\dot{w}^{-1}}$ to a $P^w\cap M_Q$-representation, given by:
\begin{align} \label{eqn isomorphism gl integral}
   f \mapsto \left(m \mapsto \delta(m)^{\frac{1}{2}}\int_{U^w} \mathrm{pr}(f(\dot{w}um))~ du \right),
\end{align}
where 
\begin{itemize}
\item $m \in M_Q$, and 
\item $\mathrm{pr}$ is the projection from $\sigma$ to $\sigma_{P\cap \dot{w}U_Q\dot{w}^{-1}}$, and 
\item $U^w=U_Q \cap \dot{w}^{-1}(U^-)\dot{w}$, and 
\item $\delta$ is the modular character of $Q$.
\end{itemize}
We remark that the above integral is well-defined as $f$ has a compact support. The modular character is multiplied to agree with the normalization factor for the Jacquet functor.


\subsection{Dual filtration I: smooth duals} \label{ss dual filtration 1}

We keep using notations in Section \ref{ss geometric lemma}. This then induces a filtration on $((\mathrm{Ind}_P^G\sigma)_{U_Q})^{\vee}$. More precisely, let $\mathcal J^{\vee}_i:=\mathcal J^{\vee}_i(Q,\sigma)$ be the subspace of $((\mathrm{Ind}_P^G\sigma)_{U_Q})^{\vee}$ precisely containing all functions from $(\mathrm{Ind}_P^G\sigma)_{U_Q}$ to $\mathbb C$ which vanish on $\mathcal J_{i+1}(Q, \sigma)$. Then we have a filtration on $((\mathrm{Ind}_P^G\sigma)_{U_Q})^{\vee}$ given by:
\[    0\subseteq \mathcal J_1^{\vee} \subseteq \ldots \subseteq \mathcal J_r^{\vee} =((\mathrm{Ind}_P^G\sigma)_{U_Q})^{\vee} .
\]

\subsection{Dual filtration II: Bernstein-Casselman pairings} \label{ss dual geometric lemma} \label{ss dual filtration 2}

We keep using notations in Section \ref{ss geometric lemma}.  Let $Q^-$ be the opposite parabolic subgroup of $Q$ and let $\bar{Q}=\dot{w}_{o,T}^{-1}(Q^-)\dot{w}_{o,T}$, which is a standard parabolic subgroup of $G$. We now construct a filtration for $(\mathrm{Ind}_P^G\sigma^{\vee})_{U_Q^-}$, as a $M_Q$-representation. 

To do so, by Lemma \ref{lem decompose by opposite},
\[  G = \bigsqcup_{u \in W_{S, T}} P\dot{u}Q^- =\bigsqcup_{u \in W_{S, T}}P\dot{u}\dot{w}_{o,T}\bar{Q}\dot{w}_{o,T}^{-1} . \]
We keep using the enumeration $u_1, \ldots, u_r$ of $W_{S, T}$ in Section \ref{ss geometric lemma}.


Let 
\[  \mathcal O_i^- =  \bigsqcup_{i'\leq i} P\dot{u}_{i'}Q^- =\bigsqcup_{i'\leq i} P\dot{v}_{i'}\bar{Q}\dot{w}_{o,T}^{-1} ,
\]
where $v_{i'}=\Omega_{S,T}(u_{i'})$.
Now, by Lemma \ref{lem algebraic closure of bwb-} (using Proposition \ref{prop order reversion}), $\mathcal O_i^-$ is open in $G$. 

We now define a filtration for $\mathrm{Ind}_P^G(\sigma^{\vee})$. Let $\mathcal I_i^-:=\mathcal I_i^-(Q^-, \sigma^{\vee})$ be the space of functions in $\mathrm{Ind}_P^G(\sigma^{\vee})$ whose support lies in $\mathcal O_i^-$. Then we have a filtration on $\mathrm{Ind}_P^G(\sigma^{\vee})$ given by:
\[ 0 \subset \mathcal I^-_1 \subset \ldots \subset \mathcal I^-_r=\mathrm{Ind}_P^G(\sigma^{\vee}) .
\]

Let $\mathrm{pr}$ be the projection from $\mathrm{Ind}_P^G(\sigma^{\vee})$ to $(\mathrm{Ind}_P^G\sigma^{\vee})_{U_{\bar{Q}}^-}$. Let $\mathcal J_i^-:=\mathcal J_i( Q^-, \sigma^{\vee})$ be the image of $\mathrm{pr}$ on $\mathcal I_i^-(Q^-,\sigma^{\vee})$. This gives a filtration for $(\mathrm{Ind}_P^G\sigma^{\vee})_{U_{\bar{Q}}^-}$:
\[   0 \subseteq \mathcal J_1^- \subseteq \ldots \subseteq \mathcal J_r^-=(\mathrm{Ind}_P^G\sigma^{\vee})_{U_{\bar{Q}}^-}.
\]

\section{Matching filtrations} \label{s match filtrations}







\subsection{Pairing for induced representations} \label{ss pairing induced}

Let $P$ be a standard parabolic subgroup of $G$. Let $\sigma$ be a smooth representation of $M_P$. Define a pairing $\langle , \rangle_{\sigma}: \sigma \times \sigma^{\vee} \rightarrow \mathbb C$ given by:
\[  \langle x, f\rangle_{\sigma}=f(x)
\]
for $x \in \sigma$ and $f \in \sigma^{\vee}$.

There is a natural isomorphism \cite[Proposition 2.25]{BZ76}:
\[ (\mathrm{Ind}_P^G\sigma)^{\vee} \cong \mathrm{Ind}_P^G(\sigma^{\vee})
\]
determined by the non-degenerate pairing on $\mathrm{Ind}_P^G\sigma \times \mathrm{Ind}_P^G(\sigma^{\vee})$ given by: for $f \in \mathrm{Ind}_P^G\sigma$ and $f' \in \mathrm{Ind}_P^G(\sigma^{\vee})$, define
\begin{align} \label{eqn pairing induced module}
  \langle f, f' \rangle = \int_{P\setminus G}  \langle f(g), f'(g) \rangle_{\sigma}~dg,
\end{align}
where $dg$ is a Haar measure on $P\setminus G$ in the sense of \cite[Section 1.20]{BZ76}.

Let $Q$ be a standard parabolic subgroup of $G$. Via Theorem \ref{thm bernstein-casselman}(1), we have:
\[  ((\mathrm{Ind}_P^G\sigma)_{U_Q})^{\vee} \cong ((\mathrm{Ind}_P^G\sigma)^{\vee})_{U_Q^-} .
\]
Combining the above two isomorphisms, we have:
\[  ((\mathrm{Ind}_P^G\sigma)_{U_Q})^{\vee} \cong (\mathrm{Ind}_P^G(\sigma^{\vee}))_{U_Q^-} 
\]
Hence, we have a non-degenerate $M_Q$-equivariant pairing $\langle , \rangle_{M_Q}: (\mathrm{Ind}_P^G\sigma)_{U_Q} \times (\mathrm{Ind}_P^G\sigma^{\vee})_{U_Q^-}$ given by: for $f \in (\mathrm{Ind}_P^G\sigma)_{U_Q}$ and $f' \in  (\mathrm{Ind}_P^G\sigma^{\vee})_{U_Q^-}$,
\begin{align} \label{eqn canonical pairing}
  \langle f, f' \rangle_{M_Q} = \langle \widetilde{f}, \widetilde{f}' \rangle ,
\end{align}
where $\widetilde{f}$ is the canonical lift of $f$ and $\widetilde{f}'$ is the canonical lift of $f'$ with respect to an open compact subgroup $K$ of $G$ such that $f$ and $f'$ are $K\cap M_Q$-invariant.

The goal of this section is to show that the filtrations in Sections \ref{ss dual filtration 1} and \ref{ss dual filtration 2} coincide.

\subsection{Canonical lifts under a filtration}
We continue to use notations in the previous section. Let $\lambda=\mathrm{Ind}_P^G(\sigma^{\vee})$. We briefly outline what we are going to do below. For $f \in \lambda_{U_Q^-}$, one suitably finds an open compact subgroup $K$ of $G$ such that $f$ is $K\cap M_Q$-invariant, finds a strictly dominant element $a \in Z(M_Q)$ and finds a sufficiently large integer $k$ to apply the stabilization theorem. Now one constructs a lift of the element $\lambda_{U_Q^-}(a^k)f$, and then one uses that to find the canonical lift of $f$ (see Lemma \ref{lem commutation rel} below), and uses Proposition \ref{lem support after acting by ak} to show $f$ satisfies some desired properties.


Let $\mathrm{pr}$ be the natural projection from $\mathrm{Ind}_P^G(\sigma^{\vee})$ to $(\mathrm{Ind}_P^G(\sigma^{\vee}))_{U_Q^-}$. Now one defines the filtrations $\mathcal I_i^-=\mathcal I_i^-(Q^-, \sigma^{\vee})$ and $\mathcal J_i^-=\mathcal J_i^-(Q^-, \sigma^{\vee})$ as in Section \ref{s geometric lemma and dual}.

\begin{lemma} \label{lem lift of geometric filt}
Let $K$ be an open compact subgroup of $G$ in a good position. Let $K_M=M_Q\cap K$. Let $f \in (\mathcal J_i^-)^{K \cap M_Q}$ (for some $i$). There exists $\widetilde{f}^* \in (\mathcal I_i^-)^K=\mathcal I^-_i\cap \lambda^K$ such that
\begin{enumerate}
\item $\mathrm{pr}(\widetilde{f}^*)$ and $f$ have the same image when projecting to the quotient $\mathcal J_i^-/\mathcal J_{i-1}^-$; and 
\item $\mathrm{supp}(\widetilde{f}^*) \subset  P\dot{w}_{i}M_Q K$.
\end{enumerate}
\end{lemma}
\begin{proof}

Let $U'{}^-=P\cap \dot{w}_iU_Q^-\dot{w}_i^{-1}$ and let $P'= \dot{w}_i^{-1}P\dot{w}_i \cap M_Q$. We first remark that we have an analogue description for (\ref{eqn isomoprhism geometric lemma}) and 
\[  C^{\infty}_c(P\dot{w}_iQ^-, \sigma)_{U_Q^-} \rightarrow \mathrm{Ind}^{M_Q}_{P'} (\sigma_{U'{}^-})^w 
\]
given by 
\begin{align} \label{eqn jacquet integral}
  f \mapsto \left(m \mapsto \delta(m)^{-\frac{1}{2}}\int_{U^{-,w_i}} f(\dot{w}_ium )~du \right) ,
\end{align}
where $\delta$ is the modular character of $Q$. 

 We now choose finitely many $m_1, \ldots, m_r$ in $M_K$ such that $P\dot{w}_im_1K, \ldots, P\dot{w}_im_rK$ are disjoint and their union covers $P\dot{w}M_Q$. This can be done since $P\setminus (P\dot{w}M_Q) \cong (P^w\cap M_Q)\setminus M_Q$ is compact and $K_M\subset K$.

Let $n_j=\dot{w}_im_j$. By Jacquet's lemma, the natural projection 
\[\mathrm{prr}_j: \sigma^{ M_P\cap n_iKn_i^{-1} } \rightarrow (\sigma_{U'{}^-})^{M_P \cap n_iK_Mn_i^{-1}} .  \]
Now, for $j=1, \ldots, r$, we define
\[  v_j = f(\dot{w}_im_j) .\]
Then we find $\widetilde{v}_j \in  \sigma^{ M_P\cap n_iKn_i^{-1} }$ such that $\mathrm{prr}_j(\widetilde{v}_j)=v_j$.

One defines $\widetilde{h}^*: G \rightarrow \sigma$ such that for all $j$, $p \in P$ and $k \in K$,
\[   \widetilde{f}^*(p\dot{w}_im_jk) =\delta(m_j)^{-\frac{1}{2}} p.v_j
\]
and $\widetilde{f}^*(g)=0$ if $g \notin P\dot{w}M_QK$. Note that the function $\widetilde{f}^*$ is well-defined from our construction, and is smooth. Now construction shows that $\widetilde{f}^*$ satisfies the property (2). Now one checks from (\ref{eqn jacquet integral}) that $\widetilde{f}^*$ satisfies the property (1). 
\end{proof}

\begin{corollary} \label{cor k invariant lift}
Let $f \in (\mathcal J_i^-)^{K_M}$ (for some $i$). There exists $\widetilde{f} \in (\mathcal I_i^-)^K=\mathcal I^-_i\cap \lambda^K$ such that 
\begin{enumerate}
\item $\mathrm{pr}(\widetilde{f})=f$;
\item $\mathrm{supp}(\widetilde{f}) \subset \bigsqcup_{i'\leq i} P\dot{w}_{i'}M_Q K$.
\end{enumerate}

\end{corollary}

\begin{proof}
We shall prove by an induction on $i$. When $i=0$, we consider $\mathcal J_0^-=0$ and $\mathcal I_0^-=0$ and so there is nothing to prove. We now consider $i \geq 1$. By Lemma \ref{lem lift of geometric filt}, there exists $\widetilde{f}^*$ such that $f-\mathrm{pr}(\widetilde{f}^*) \in \mathcal J_{i-1}^-$ and $\mathrm{supp}(\widetilde{f}^*)\subset P\dot{w}_{i}M_QK$. By induction, we have an element $\widetilde{f}' \in (\mathcal I_{i-1}^-)^K$ such that $\mathrm{pr}(\widetilde{f}')=f-\mathrm{pr}(\widetilde{f}^*)$, and $\mathrm{supp}(\widetilde{f}')\subset \bigsqcup_{i'\leq i-1}P\dot{w}_{i'}M_QK$. Now, the element $\widetilde{f}'+\widetilde{f}^*$ provides the desired element.
\end{proof}

We shall need the following formula. \\

\begin{lemma} (c.f. \cite[Proposition 5.2]{Be87}) \label{lem commutation rel}
Let $K$ be an open compact subgroup of $G$ in a good position. Let $\widetilde{f} \in \lambda^K$. Then $(\lambda_{U_Q^-})(a^{-k}) \circ \mathrm{pr}(\widetilde{f}) =\mathrm{pr}\circ \lambda_K(a^{-k})(\widetilde{f})$.
\end{lemma}
\begin{proof}
By the Iwahori decomposition of $K$, 
\[   K= (U_Q^-\cap K)\cdot (M_Q \cap K) \cdot (U_Q \cap K) .\]
Note that $a^k(U_Q \cap K)a^{-k} \subset K$ by the strict dominance of $a$. We also have $a^k(M_Q\cap K)a^{-k} \subset K$ and so:
\[   \lambda_K(a^{-k})(f) =e_{U_Q^-\cap K} \circ \lambda(a^{-k}) \circ e_K(f) =e_{U_Q^-\cap K} \circ \lambda(a^{-k})(f),
\] 
where $e_{K'}(f)(g)=\int_{h \in K'} f(gh) ~dh$ for $K'=U_{Q}^-\cap K$ or $ K$. (Here $dh$ is a fixed measure with $\int_{h \in K'} \mathrm{ch}_{K'}(h)~dh=1$.) We also remark that $\lambda(a^{-k})$ are simply the group actions of $a^{-k}$ on $\lambda$. Since $e_{U_Q^-\cap K}$ coincides with taking the $U_Q^-\cap K$-coinvariants and $\mathrm{pr}$ coincides with taking $U_Q^-$-coinvaraints, we have that $\mathrm{pr}\circ e_{U_Q^-\cap K}=\mathrm{pr}$. Now the lemma follows by the commutation between $\mathrm{pr}$ and $\lambda(a^{-k})$.

\end{proof}

\begin{proposition} \label{prop canoniclal lift geom filt}
Let $f \in \mathcal J_i^-$. Let $K$ be an open compact subgroup of $G$ such that $f$ is $K \cap M_Q$-invariant (i.e. $f \in (\lambda_{U_Q^-})^{K \cap M_Q}$). Let $K_M=K\cap M_Q$. Let $a$ be a strictly dominant element for $M_Q$. Let $k$ be a sufficiently large integer such that $\lambda_K(a^{-k})$ on $\lambda^K$ is stable. The canonical lift $\widetilde{f}$ of $f$ is in $(\mathcal I_i^-)^K$. Moreover, the canonical lift $\widetilde{f}$ satisfies the property that $\lambda_K(a^{-k})\widetilde{f}\in (\mathcal I_i^-)^K$.
\end{proposition}

\begin{proof}
Let $f'=a^k.f$. Since $a \in Z(M_Q)$ and $f \in (\lambda_{U_Q^-})^{K_M}$, $f' \in (\lambda_{U_Q^-})^{K_M}$. By Corollary \ref{cor k invariant lift}, there exists a lift $\widetilde{f}' \in (\mathcal I_i^-)^K$ of $f'$ satisfying the property (2) in Corollary \ref{cor k invariant lift}. By Lemma \ref{lem commutation rel}, $\mathrm{pr}(\lambda_K(a^{-k})\widetilde{f}')$ is a lift of $f$, but then $\lambda_K(a^{-k})\widetilde{f}'$ is the canonical lift of $f$ by the definition. Now it follows from Proposition \ref{lem support after acting by ak} and Property (2) in Corollary \ref{cor k invariant lift} (for $\widetilde{f}'$) that the canonical lift of $f$ is in $\mathcal I^-_i$. The second assertion follows from the first assertion and $\lambda_K(a^{-i})\lambda_K(a^{-i})=\lambda_K(a^{-2i})$, see e.g. \cite[Lemma 4.1.5]{Ca93}.
\end{proof}


\subsection{Constructing a  lift for $\mathcal J_{i+1}$}





We keep using notations in the previous two subsections. Let $\pi =\mathrm{Ind}_Q^G\sigma$. Let $\mathcal I_i=\mathcal I_i(Q, \sigma)$ and $\mathcal J_i=\mathcal J_i(Q,\sigma)$ as in Section \ref{s geometric lemma and dual}. Note that $\pi^{\vee}=\lambda$. We now have the following lemma for constructing a lift of an element in $\mathcal I_i$. 

\begin{lemma} \label{lem find disjoint support f}
Let $f \in \mathcal J_{i+1}$. Let $K$ be an open compact subgroup of $G$ in a good position such that $f \in (\lambda_{U_Q})^{K_M}$. Let  $\mathcal S\subset \mathcal O_{i}^-$ such that $\mathcal S$ is invariant under right $K$-action i.e.
\[   \left\{ gk : g \in \mathcal S, k \in K\right\} =\mathcal S. 
\]
There exists $\widetilde{f} \in \mathcal I_{i+1}^K$ such that $\mathrm{pr}(\pi_K(a^k)\widetilde{f})=f$ and
\[\mathrm{supp}(\widetilde{f}) \subset \bigsqcup_{i' \geq i+1} P\dot{w}_{i'}M_QK .\]
Moreover,
\[  \mathcal S \cap \mathrm{supp}(\widetilde{f}) =\emptyset .
\]

\end{lemma}

\begin{proof}
A construction for $\widetilde{f}$ satisfying the support property (i.e. the former property) is similar to the one in Corollary \ref{cor k invariant lift}. 

We now consider that the intersection property (i.e. the latter proeprty) follows from the first one. Indeed, 
\[ \mathcal O_i^-\cap  \left(\bigsqcup_{i'\geq i+1} P\dot{w}_{i'}M_Q \right)=\emptyset, \quad \mbox{ and so } \quad \mathcal S \cap \left(  \bigsqcup_{i'\geq i+1} P\dot{w}_{i'}M_Q \right)=\emptyset  ,
\]
where the first isomorphism uses $P\dot{w}_{i'}M_Q \subset P\dot{w}_{i'}Q^-$ and the opposite Bruhat decomposition. Since $\mathcal S$ is invariant under right $K$-action, we then also have that
\[ \mathcal S \cap  \left(\bigsqcup_{i'\geq i+1} P\dot{w}_{i'}M_QK \right) =\emptyset  .
\]
Now the latter property follows from the former one.
\end{proof}

\subsection{Coincidence of filtrations}

\begin{theorem} \label{thm filtrations coincide bc pairing}
Let $P, Q$ be standard parabolic subgroups of $G$. Let $\sigma$ be a smooth representation of $M_P$. We write $\mathcal J_i^{\vee}(Q, \sigma)$ for the filtration on $((\mathrm{Ind}_P^G\sigma)_{U_Q})^{\vee}$ as in Section \ref{ss dual filtration 1}
 with a fixed ordering on $W_{S,T}$. Write $\mathcal J^-_i(Q^-,\sigma^{\vee})$ for the filtration on $(\mathrm{Ind}_P^G\sigma^{\vee})_{U_Q^-}$ as in Section \ref{ss dual geometric lemma}. Then the two filtrations coincide via the Bernstein-Casselman isomorphism in Theorem \ref{thm bernstein-casselman}(1). 

\end{theorem}





The proof of Theorem \ref{thm filtrations coincide bc pairing} is divided into two steps:
\begin{enumerate}
\item We first show that $\langle , \rangle_{M_Q}$ is identically zero on $\mathcal J_{i+1}\times \mathcal J_i^-$.
\item  We next show a non-degeneracy of $\langle , \rangle_{M_Q}$ when restricted to $(\mathcal J_i/ \mathcal J_{i+1}) \times \mathcal J^-_i$.
\end{enumerate}

Let $\pi=\mathrm{Ind}_P^G\sigma$ and let $\lambda=\mathrm{Ind}_P^G(\sigma^{\vee})$ (in the notation of Theorem \ref{thm filtrations coincide bc pairing}). We first prove Step 1.



\begin{lemma} \label{lem pairing radical}
Let $f \in \mathcal J_{i+1}$ and let $f' \in \mathcal J_i^-$. Then $\langle f, f' \rangle_{M_Q}=0$.
\end{lemma}

\begin{proof}
Let $K$ be an open compact subgroup such that $f$ and $f'$ are $K\cap M_Q$-invariant. Let $\widetilde{f}$ (resp. $\widetilde{f}'$) be the canonical lift of $f$ (resp. $f'$) in $\pi^K$ (resp. $\lambda^K)$. It follows from Theorem \ref{thm bernstein-casselman}(2) that
\[  \langle f, f' \rangle_{M_Q} =\langle \widetilde{f}, \widetilde{f}' \rangle .
\]
Let $a$ be a strictly dominant element for $M_Q$, and let $k$ be a sufficiently large integer such that $\pi_K(a^k)$ is stable. By Lemma \ref{lem find disjoint support f}, $\widetilde{f}=\pi_K(a^k)\widetilde{h}$ for some $\widetilde{h} \in \pi^K$ with properties 
\[  \mathrm{supp}(\widetilde{h}) \subset \bigsqcup_{i'\geq i+1} P\dot{w}_{i'}M_QK
\]
Thus, now we also have that
\begin{align}
\langle \widetilde{f}, \widetilde{f}' \rangle &= \langle \pi_K(a^k)\widetilde{h}, \widetilde{f}'\rangle \\
                                             &= \langle \widetilde{h}, \pi_K(a^{-k})\widetilde{f} \rangle  \\
																						 & = \int_{P\setminus G}  \langle \widetilde{h}(g) ,   (\pi_K(a^{-k})\widetilde{f})(g)\rangle_{\sigma} ~dg \\
																						 &= 0
\end{align}
where the first three equalities follow from definitions, and the last equality follows from $\mathrm{supp}(\pi_K(a^{-k})\widetilde{f}) \subset \mathcal I_i^-$ in Proposition \ref{prop canoniclal lift geom filt} and so $\mathrm{supp}(\pi_K(a^{-k})\widetilde{f}) \cap \mathrm{supp}(\widetilde{h})=\emptyset$ by Lemma \ref{lem find disjoint support f}. 

Now the lemma follows from combining the above two equations.
\end{proof}

We now prove Step 2. Let $\mathcal J_i^{\vee}=\mathcal J_i^{\vee}(Q,\sigma)$.




\begin{lemma} \label{lem non degener j- version}
Let $f' \in \mathcal J^-_i$, which does not lie in $\mathcal J_{i-1}^-$. Let $K$ be an open compact subgroup of $G$ in a good position such that $f'$ is $K\cap M_Q$-invariant. Let $a$ be a strictly dominant element and let $k$ be a sufficiently large integer such that $\lambda_K(a^{-k})$ is stable. Then there exists $f \in \mathcal J_i$ such that $\langle f, f' \rangle \neq 0$.
\end{lemma}

\begin{proof}
Since $f' \in \mathcal J^-_i-\mathcal J_{i-1}^-$, the canonical lift  $\widetilde{f}'$ of $f'$ and $\lambda_K(a^{-k})\widetilde{f}'$ lie  in $\mathcal I^-_i -\mathcal I^-_{i-1}$ by Proposition \ref{prop canoniclal lift geom filt}. Let $\widetilde{h}=\lambda_K(a^{-k})\widetilde{f}'$. Note that $\mathrm{pr}(\widetilde{h})\not\in \mathcal J_{i-1}^-$ by a similar argument in Lemma \ref{lem commutation rel}. Thus we have $\widetilde{h}(\dot{w}_jq)\neq 0$ for some $q \in Q$. We now have that $\widetilde{h}$ is constant in $\dot{w}_jqK$. 

Now we let $y=\widetilde{h}(\dot{w}_jq)$ and let $x$ be an element in $\sigma^{\vee}$ such that $\langle x, y \rangle_{\sigma}=1$. Let $\widetilde{f}$ be the function in $\mathrm{Ind}_P^G\sigma$ with $\mathrm{supp}(\widetilde{f})=P\dot{w}_jqK$ and $\widetilde{f}(\dot{w}_jqK)=x$. Then 
\[  \langle \pi_K(a^k)\widetilde{f}, \widetilde{f}' \rangle =\langle \widetilde{f}, \lambda_K(a^{-k})\widetilde{f}' \rangle = \langle \widetilde{f}, \widetilde{h} \rangle =\mathrm{vol}(P\setminus P\dot{w}_jqK) \neq 0 ,
\]
where the second equality follows from our construction of $\widetilde{f}$ and (\ref{eqn pairing induced module}).  This then gives that $\langle \mathrm{pr}(\widetilde{h}), f') \rangle_{M_Q}\neq 0$ by (\ref{eqn canonical pairing}) as desired.
\end{proof}

\noindent
{\it Proof of Theorem \ref{thm filtrations coincide bc pairing}.} By Lemma \ref{lem pairing radical}, $ \mathcal J^-_i \subset \mathcal J_i^{\vee}$. We now prove another inclusion. Let $f' \in \mathcal I_i^{\vee}$. If $f' \not\in \mathcal I^-_i$, then $f' \in \mathcal I^-_{j}$ for some $j>i$ and does not lie in $\mathcal I^-_i$. Let $\widetilde{f}'$ be the canonical lift of $f'$. By Proposition \ref{prop canoniclal lift geom filt}, $\widetilde{f}' \in \mathcal I^-_j$. Since $\widetilde{f}'$ does not lie in $\mathcal I^-_i$, its canonical lift also does not lie in $\mathcal I^-_i$. By Lemma \ref{lem non degener j- version}, there exists $\widetilde{f} \in \mathcal J_k$ for $j\geq k >i$ such that $\langle \widetilde{f}, \widetilde{f}' \rangle \neq 0$. Thus, $\langle f, f' \rangle_{M_Q} =\langle \widetilde{f}, \widetilde{f}' \rangle \neq 0$, where $f$ is the projection of $\widetilde{f}$ from $\mathrm{Ind}_P^G\sigma$ onto $(\mathrm{Ind}_P^G\sigma)_{U_Q}$. By definition, $f$ is also in $\mathcal I_k \subset \mathcal I_i$. This contradicts that $f' \in \mathcal I_i^{\vee}$. Hence, $f' \in \mathcal I^-_i$ as desired. \qed

\vspace{1cm}

\begin{remark}
Let $p$ be the residual characteristics of $F$. Let $R$ be a noetherian $\mathbb Z[\frac{1}{p}]$-ring. A recent paper of Dat-Helm-Kurinczuk-Moss \cite{DHKM23} establishes the second adjointness theorem for the category of smooth $R$-representations. Indeed, they also established the stabilization theorem in \cite[Corollary 4.9]{DHKM23}. One may expect that an analogue of Theorem \ref{thm filtrations coincide bc pairing} can be generalized to such category.
\end{remark}


\section{Applications} \label{s consequences}

\subsection{Compositions of maps from inductions}

Let $P$ be a standard parabolic subgroup of $G$. Then, for a smooth representation $\lambda$ of $P$, one defines $\lambda_{U_P}$ to be the space:
\[  \lambda_{U_P} := \delta_{P}^{-\frac{1}{2}}\cdot \frac{\lambda}{\langle u.x-x : x \in \lambda, u \in U_P \rangle} ,
\]
viewed as an $M_P$-representation. This is the same as the one defined in Section \ref{s jacquet functor}, but we emphasis here that we are considering a $P$-representation. Similar notion $\lambda_{U_P^-}$ is used for a smooth $P^-$-representation $\lambda$. Note that such functor is exact in the category of smooth representations of $P$ (see e.g. \cite[Proposition 2.35]{BZ76}).

\begin{corollary} \label{cor composition non-zero in induced rep}
Let $P$ be a standard parabolic subgroup of $G$. Let $\sigma$ be a smooth representation of $M_P$. Let $\pi$ be a non-zero quotient of $\mathrm{Ind}_P^G\sigma$. Let $\lambda=C^{\infty}_c(PQ^-, \sigma)$ be the subspace of $\mathrm{Ind}_P^G\sigma$ containing all functions supported in $PQ^-$, which is viewed as a $Q^-$-module via right translation. Then
\begin{itemize}
\item[(1)] the natural composition $\lambda_{U_Q^-} \hookrightarrow (\mathrm{Ind}_P^G\sigma)_{U_Q^-} \twoheadrightarrow \pi_{U_Q^-}$ is non-zero;
\item[(2)] if $P=Q$, then the map in (a) takes the form: 
\[  \sigma \hookrightarrow (\mathrm{Ind}_P^G\sigma)_{U_P^-} \twoheadrightarrow \pi_{U_P^-}. \]

\end{itemize}
\end{corollary}

\begin{proof}
We first prove (1). By taking the dual on the surjection $\mathrm{Ind}_P^G\sigma \twoheadrightarrow \pi$, we have an embedding:
\[   \pi^{\vee} \hookrightarrow \mathrm{Ind}_P^G\sigma^{\vee} 
\]
Taking the Jacquet functor gives:
\begin{align} \label{composition of maps}
   (\pi^{\vee})_{U_Q} \hookrightarrow (\mathrm{Ind}_P^G\sigma^{\vee})_{U_Q} \twoheadrightarrow C^{\infty}_c(PQ, \sigma^{\vee})_{U_Q}
\end{align}
which is non-zero. 

Since the map in Theorem \ref{thm bernstein-casselman}(1) is functorial, the first injection in (\ref{composition of maps}) takes the following form under the isomorphism of Theorem \ref{thm bernstein-casselman}(1):
\[  (\pi_{U_{Q}^-})^{\vee} \hookrightarrow ((\mathrm{Ind}_P^G\sigma)_{U_{Q}^-})^{\vee} 
\]
and is obtained by first taking the Jacquet functor with respect to $Q^-$ on the projection $\mathrm{Ind}_P^G\sigma$ and then taking the smooth dual.

By Theorem \ref{thm filtrations coincide bc pairing}, the second map coincides with the embedding $\lambda_{U_Q^-} \hookrightarrow (\mathrm{Ind}_P^G\sigma)_{U_Q^-}$. Thus  (\ref{composition of maps}) is obtained by taking the smooth dual on $\lambda_{U_Q^-} \hookrightarrow (\mathrm{Ind}_P^G\sigma)_{U_Q^-} \twoheadrightarrow \pi_{U_Q^-}$.  Then we must have the non-zero composition in the first assertion.

We now prove (2). Note that $C^{\infty}_c(PP^-, \sigma)$ is naturally isomorphic to the space of smooth functions from $U_P^-$ to $\sigma$. Taking the Jacquet functor with respect to $U_P^-$ yields an isomorphism to $\sigma$. One checks the $M_P$-action is isomorphic to $ \sigma$.
\end{proof}

The above shows a phenomenon that is probably not expected by many people at the first glance, and we provide an example to illustrate more details. 

\begin{example}
Let us illustrate Corollary \ref{cor composition non-zero in induced rep} with an example. We consider $G=\mathrm{GL}_2(F)$. Let $B$ be the subgroup of upper triangular matrices in $G$. Let $U$ be the subgroup of unipotent upper triangular matrices in $G$, and let $U^-$ be the subgroup of unipotent lower triangular matrices in $G$. The subgroup $T$ of diagonal matrices is naturally identified with $F^{\times} \times F^{\times}$ via the following map:
\[  \begin{pmatrix}  a_1 &  \\ & a_2 \end{pmatrix} \mapsto (a_1, a_2) .
\]
Let $\nu:=|.|_F$ be the normalized absolute value on $F$. Define $\chi: F^{\times} \times F^{\times}\rightarrow \mathbb C$ given by $\chi((a_1, a_2))=\nu(a_1)^{\frac{1}{2}} \cdot \nu(a_2)^{-\frac{1}{2}}$. 

Let $\mathbb C_{\mathrm{triv}}$ be the trivial representation of $\mathrm{GL}_2(F)$, and let $\mathrm{St}$ be the Steinberg representation of $\mathrm{GL}_2(F)$. It is well-known that
\[  (\mathbb C_{\mathrm{triv}})_{U^-} \cong \chi, \quad \mathrm{St}_{U^-} \cong \chi^{-1} .
\]
The parabolically induced representation $\mathrm{Ind}_B^G\chi$ has a unique simple quotient, and the simple quotient is the trivial representation $\mathbb C_{\mathrm{triv}}$. Then the quotient map $\mathrm{Ind}_B^G\chi \rightarrow \mathbb C_{\mathrm{triv}}$ induces the surjection in Corollary \ref{cor composition non-zero in induced rep}(2). On the other hand, the space $C^{\infty}_c(BB^-, \chi)_{U_B^-}$ is just isomorphic to $\chi$, as $T$-representations. Since $(\mathrm{Ind}_B^G\chi)_{U_B^-}$ contains only $\chi$ and $\chi^{-1}$ (and $\chi\not\cong \chi^{-1}$), we must have that composition in Corollary \ref{cor composition non-zero in induced rep}(2) is non-zero by simply comparing the characters in the spaces.

Indeed, there is a subtly when one first looks at this example. If one considers the composition factors of $\mathrm{Ind}_B^G\chi$, then one has the short exact sequence:
\[  0 \rightarrow \mathrm{St} \rightarrow \mathrm{Ind}_B^G\chi \rightarrow \mathbb{C}_{\mathrm{triv}} \rightarrow 0 . \]
Taking the Jacquet functor $U^-$, one has:
\begin{align} \label{eqn first filtration}
 0 \rightarrow \chi^{-1} \rightarrow (\mathrm{Ind}_B^G\chi)_{U^-}  \rightarrow  \chi \rightarrow 0 .
\end{align}

On the other hand, if one considers the filtration from the double cosets in $B\setminus G/B^-$ (the one in Section \ref{ss dual geometric lemma}), one has:
\[   C^{\infty}_c(B\setminus BB^-, \chi)_{U^-} \cong \chi, \quad C^{\infty}_c(B\setminus B\begin{pmatrix} 0 & 1 \\ -1 & 0 \end{pmatrix}, \chi)_{U^-}\cong \chi^{-1} .
\]
 One then obtains the short exact sequence:
\begin{align} \label{eqn second filtration}
  0 \rightarrow \chi \rightarrow  (\mathrm{Ind}_B^G\chi  )_{U^-} \rightarrow \chi^{-1} \rightarrow 0 .
\end{align}
We emphasis that the normalizing factor is important in the computation. 

Combining the right part of former sequence and the left part of latter seuqence, one has the map in Corollary \ref{cor composition non-zero in induced rep}:
\[ 0 \rightarrow \chi \rightarrow  (\mathrm{Ind}_B^G\chi  )_{U^-} \rightarrow \chi \rightarrow 0 .
\]
This implies that $\chi$ splits in the sequence and $\chi$ is a direct summand in $(\mathrm{Ind}_B^G\chi  )_{U^-}$, as a consequence.

These two filtrations (\ref{eqn first filtration}) and (\ref{eqn second filtration}) are opposites, and indeed, this suggests that one may obtain additional information when attempting to combine them. This illustrates a key aspect discussed in \cite{Ch22+, Ch22+b}, particularly in cases where the injection in Corollary \ref{cor composition non-zero in induced rep}(2) splits. The same idea is also applied to graded Hecke algebras to investigate a splitting of induced modules in \cite[Section 10]{CH25+}. Additionally, we note that a dual version can be considered in Lemma \ref{lem realize adjoint map} below.
\end{example}

\subsection{The adjoint map} \label{ss adjointness map}

\begin{lemma} \label{lem embedding by double duals}
Let $\pi_1, \pi_2$ be smooth representations of $G$. Let $f: \pi_1 \rightarrow \pi_2^{\vee}$. Let $f': \pi_2 \rightarrow \pi_1^{\vee}$ be the map given by $f'(x)(y)=f(y)(x)$ for $x \in \pi_2$ and $y \in \pi_1$. Let $f^{\vee}: (\pi_2^{\vee})^{\vee} \rightarrow \pi_1^{\vee}$ given by $f^{\vee}(h)(y)=h(f(y))$ for $h \in (\pi_2^{\vee})^{\vee}$ and $y \in \pi_1$. Let $\iota: \pi_2 \rightarrow (\pi_2^{\vee})^{\vee}$ be the embedding given by: for $x \in \pi_2$ and $f\in \pi_2^{\vee}$, $\iota(x)(f)=f(x)$. Then $f'=f^{\vee}\circ \iota$.
\end{lemma}

The above lemma is straightforward and we omit the details. The following lemma is a dual version of Corollary \ref{cor composition non-zero in induced rep}(2).

\begin{lemma} \label{lem realize adjoint map}
Let $P$ be a standard parabolic subgroup of $G$. Let $\sigma$ be a smooth representation of $M_P$. Let $f \in \mathrm{Hom}_G(\pi, \mathrm{Ind}_P^G\sigma)$. Via Frobenius reciprocity, the corresponding map in $\mathrm{Hom}_{M_P}(\pi_{U_P}, \sigma)$ is given by the composition:
\[   \pi_{U_P} \stackrel{\widetilde{f}}{\rightarrow} (\mathrm{Ind}_P^G\sigma)_{U_P} \twoheadrightarrow C^{\infty}_c(P, \sigma)_{U_P} \cong \sigma ,
\]
where $\widetilde{f}$ is descended from $f$ by taking the Jacquet functor.
\end{lemma}

\begin{proof}
Recall that the adjoint map $F$ of $f$ is given by 
\[   F(x)=f(x)(1),
\]
where $1$ is the identity element in $G$ \cite{BZ76}. Note that the projection $\mathrm{Ind}_P^G\sigma$ to $C^{\infty}_c(P, \sigma)$ coincides with the evaluation map of a function of $\mathrm{Ind}_P^G\sigma$ at the identity. Thus the two maps coincide.
\end{proof}

We now recall the relation between the second adjointness theorem and the Bernstein-Casselman pairing. Let $\pi$ be a smooth representation of $G$ and let $\sigma$ be a smooth representation of $M_P$. It follows from \cite[Page 63]{Be92} that 
\begin{align*}
 \mathrm{Hom}_G(\mathrm{Ind}_P^G\sigma, \pi^{\vee}) &\cong \mathrm{Hom}_G(\pi, (\mathrm{Ind}_P^G\sigma)^{\vee}) \\
                                                    &\cong \mathrm{Hom}_G(\pi, \mathrm{Ind}_P^G(\sigma^{\vee})) \\
																										&\cong \mathrm{Hom}_G(\pi_{U_P}, \sigma^{\vee}) \\
																										&\cong \mathrm{Hom}_{M_P}(\sigma, (\pi_{U_P})^{\vee}) \\
																										&\cong \mathrm{Hom}_{M_P}(\sigma, (\pi^{\vee})_{U_P^-}),
\end{align*}
where the first ismorphism follows by taking dual maps, the second isomorphism follows from \cite{BZ76}, the third isomorphism follows from Frobenius reciprocity, the fourth isomorphism follows from taking dual maps again and the last isomorphism follows from Theorem \ref{thm bernstein-casselman}(1). The above isomorphism determines the second adjointness theorem.

We now use our result to recover a result of Bezrukavnikov-Kazhdan.

\begin{corollary} \label{cor embedding from second adjointness} \cite{BK15}
Let $\pi$ be a smooth representation of $G$. Let $P$ be a standard parabolic subgroup of $G$ and let $\sigma$ be a smooth representation of $M_P$. Let $f \in \mathrm{Hom}_G(\mathrm{Ind}_P^G\sigma, \pi)$.  Let $U^-=U^-_P$. Then, the map $\sigma \hookrightarrow \pi_{U^-}$ arising from the adjointness coincides with taking the Jacquet functor $U^-$ on the following composition of maps:
\[ C^{\infty}_c(PP^-, \sigma) \hookrightarrow \mathrm{Ind}_P^G\sigma \twoheadrightarrow \pi 
\]
with the ismorphism $C^{\infty}_c(PP^-,\sigma)_{U^-}\cong  \sigma$.
\end{corollary}

\begin{proof}

Let $U=U_P$. We begin with a non-zero map $f$ in $ \mathrm{Hom}_{G}(\mathrm{Ind}_P^G\sigma, \pi^{\vee})$. Then let $f^{\vee}$ be the corresponding dual map in $\mathrm{Hom}_G(\pi, (\mathrm{Ind}_P^G\sigma)^{\vee})$. By Lemma \ref{lem realize adjoint map}, the adjoint map of $f$ in $\mathrm{Hom}_{G}(\pi_U, \sigma)$ is given by:
\[  \pi_U \hookrightarrow (\mathrm{Ind}_P^G\sigma^{\vee})_U \twoheadrightarrow C^{\infty}_c( P, \sigma^{\vee})_U \cong \sigma^{\vee},
\]
where the maps are natural maps in Lemma \ref{lem realize adjoint map}. 

Taking the dual and the Bernstein-Casselman isomorphism gives the following non-zero composition of maps:
\begin{align} \label{eqn compose second adjoint}
   (\pi^{\vee})_{U^-} \leftarrow  ((\mathrm{Ind}_P^G\sigma^{\vee})^{\vee})_{U^-} \hookleftarrow  C^{\infty}_c(PP^-, (\sigma^{\vee})^{\vee})_{U^-} \cong (\sigma^{\vee})^{\vee},
\end{align}
where the first map follows from the functorial isomorphism of Bernstein-Casselman and the second map follows from our matching of filtrations in Theorem \ref{thm filtrations coincide bc pairing}. 

In order to obtain the adjoint map in $\mathrm{Hom}_{M_P}(\sigma, (\pi^{\vee})_{U^-})$, we then compose (\ref{eqn compose second adjoint}) with the natural embedding $\sigma \hookrightarrow \sigma^{\vee}{}^{\vee}$. By functoriality, we also have the following commutative diagram:
\[ \xymatrix{  (\pi^{\vee})_{U^-}   &  \ar@{->>}[l]_{h} (\mathrm{Ind}_P^G(\sigma^{\vee}{}^{\vee}))_{U^-}  &  \ar@{_{(}->}[l] C^{\infty}_c( PP^-, \sigma^{\vee}{}^{\vee})_{U^-} \cong \sigma^{\vee}{}^{\vee} \\
              &     \ar@{->}[u]^{s} (\mathrm{Ind}_P^G\sigma)_{U^-} & \ar@{_{(}->}[l]^{i}  \ar@{->}[u] C^{\infty}_c( PP^-, \sigma)_{U^-}\cong \sigma },
\]
where $h$ is the natural map obtained from $f$ via taking dual, Jacquet functors and the Bernstein-Casselman pairing, and $i$ is the map induced from the embedding as a subspace.

The natural map from $\mathrm{Ind}_P^G\sigma$ to $\mathrm{Ind}_P^G(\sigma^{\vee}{}^{\vee})$ coincides with the natural embedding $\mathrm{Ind}_P^G\sigma$ to $(\mathrm{Ind}_P^G\sigma)^{\vee}{}^{\vee}$ under the isomorphism between $\mathrm{Ind}_P^G(\sigma^{\vee}{}^{\vee})$ and $(\mathrm{Ind}_P^G\sigma)^{\vee}{}^{\vee}$,  by Lemma \ref{lem embedding by double duals}. Thus, by Lemma \ref{lem realize adjoint map}, $h \circ s$ is obtained from $f$ via taking the Jacquet functor. Using the composition $h\circ s \circ i$, we obtain the desired description of the adjoint map of $f$ in the case that $f$ is in $\mathrm{Hom}_G(\mathrm{Ind}_P^G\sigma, \pi^{\vee})$.

We now consider general $\pi$ (not necessarily of the form $\pi^{\vee}$). Let $f \in \mathrm{Hom}_G(\mathrm{Ind}_P^G\sigma, \pi)$. One considers the following embedding:
\[  0 \rightarrow \pi \rightarrow (\pi^{\vee})^{\vee} .
\]
Hence, we have an injection from $\mathrm{Hom}_G(\mathrm{Ind}_P^G\sigma, \pi)$ to $\mathrm{Hom}_G(\mathrm{Ind}_P^G\sigma, \pi^{\vee}{}^{\vee})$. Hence, $f$ descends to a map, denoted by $\bar{f}$, in  $\mathrm{Hom}_G(\mathrm{Ind}_P^G\sigma, \pi)$. 
By functoriality of the second adjointness, we then have the following commutative diagram
\[  \xymatrix{  \pi_{U^-} \ar[d]   &  & \ar[ll]_{\bar{f}} \sigma \ar@{=}[d] \\ 
                (\pi^{\vee}{}^{\vee})_{U^-} & \ar[l] (\mathrm{Ind}_P^G\sigma)_{U^-} & \ar[l]  C^{\infty}_c(PP^-, \sigma)_{U^-}  } ,
\]
where the bottom horizontal maps are the maps whose composition gives the adjoint map of $\bar{f}$ shown above, the left vertical map is induced from the natural embedding $\pi$ to $\pi^{\vee}{}^{\vee}$. This then gives the desired description of the map $\bar{f}$ in general case.
\end{proof}

\subsection{Applications to derivatives for $\mathrm{GL}$} \label{ss appl derivatives}

Let $G_n=\mathrm{GL}_n(F)$ be the general linear group over a non-Archimedean local field $F$. Let $P_{n_1,n_2}$ be the parabolic subgroup containing $\mathrm{diag}(g_1, g_2)$ for $g_1 \in G_{n_1}$ and $g_2 \in G_{n_2}$ and upper triangular matrices.  For smooth representations $\pi_1$ and $\pi_2$ of $G_{n_1}$ and $G_{n_2}$ respectively, we shall denote by $\pi_1 \times \pi_2$ the normalized parabolic induction $\mathrm{Ind}^{G_{n_1+n_2}}_{P_{n_1,n_2}} (\pi_1 \boxtimes \pi_2)$.

For a smooth representation $\pi$ of $G_n$, we shall denote the 'left' $i$-th Bernstein-Zelevinsky derivative by ${}^{(i)}\pi$, which is a $G_{n-i}$-representation. We shall not recall the definitions of those derivatives and one may see e.g. \cite{BZ77, Ch22+c}. 

The level of an irreducible representation $\pi$ of some $G_n$ is the largest integer $i^*$ such that ${}^{(i^*)}\pi\neq 0$. It is known that in such case, ${}^{(i^*)}\pi$ is irreducible, and we shall also write ${}^-\pi ={}^{(i^*)}\pi$ (usually referred as the left Bernstein-Zelevinsky highest derivative of $\pi$).

\begin{corollary} \label{cor descending derivatives}
Let $\tau$ be an irreducible representation of $G_n$ and let $\sigma$ be a smooth representation of $G_k$. Let $\pi$ be an irreducible quotient of $\sigma \times \tau$. Suppose the level of $\pi$ is equal to the level of $\tau$. Then ${}^-\pi$ is an irreducible quotient of $\sigma \times {}^-\tau$. 
\end{corollary}

\begin{proof}
Let $l$ be the level of $\pi$. Let $U^-$ be the opposite unipotent radical in $P_{k,n}$. It follows from Corollary \ref{cor composition non-zero in induced rep} that we have a non-zero map $\sigma \boxtimes \tau \rightarrow \pi_{U^-}$ from the second adjointness theorem on the quotient map from $\sigma \times \tau$ to $\pi$, and so we have a non-zero map
\[ \sigma \boxtimes \tau \twoheadrightarrow \sigma' \boxtimes \tau \hookrightarrow \pi_{U^-}  \]
for some quotient $\sigma'$ of $\sigma$. On the factor $G_n$ in $G_k \times G_n$, we consider the unipotent subgroup $R$ containing the elements:
\begin{align} \label{eqn defining group for derivatives}
  \begin{pmatrix} I_{n-l} & \\   m & u \end{pmatrix} \subset G_n,
\end{align}
where $m \in \mathrm{Mat}_{l, n-l}$ and $u \in U_l^-$. Here $U_l^-$ is the subgroup of $G_l$ containing all unipotent lower triangular matrices. Let $\psi: R \rightarrow \mathbb C$ be given by $\psi(\begin{pmatrix} I_{n-l} & \\ m & u \end{pmatrix}) =\bar{\psi}(u)$, where $\bar{\psi}$ is a non-degenerate character on $U^-_l$. Thus, by exactness, we obtain a non-zero map:
\[  \sigma \boxtimes \tau_{R, \psi} \twoheadrightarrow \sigma' \boxtimes \tau_{R, \psi} \hookrightarrow (\pi_{U^-})_{R,\psi} .
\]

Now, one defines similarly $R'$ in $G_{n+k}$ by replacing $I_{n-l}$ with $I_{n+k-l}$ and replacing $m \in \mathrm{Mat}_{l,n-l}$ with elements in $\mathrm{Mat}_{l,n+k-l}$. We similarly have a character $\psi'$ on $R'$. 

Now let $\widetilde{R}$ (resp. $\widetilde{\widetilde{R}}$) be the subgroup containing matrices of the form 
\[ \begin{pmatrix} g_1 &  &  \\  * & g_2 &  \\ * & * & u \end{pmatrix} \quad \mbox{(resp. $\begin{pmatrix} g & \\ * & u \end{pmatrix}$)} ,\] 
where $g_1 \in G_k$, $g_2 \in G_{n-l}$ and $u \in U_l^-$ (resp. $g \in G_{n+k-l}$ and $u \in U_l^-$). By taking Jacquet functors in stage, we have a non-zero map
\[  C^{\infty}_c(P\widetilde{R}, \sigma\boxtimes \tau)_{R', \psi'} \hookrightarrow (\sigma\times\tau)_{R',\psi'} \rightarrow  \pi_{R', \psi} .
\]
Since the map $C^{\infty}_c(P\widetilde{R}, \sigma\boxtimes \tau)$ to $\sigma \times \tau$ factors through the injection from $C^{\infty}_c(P\widetilde{R}, \sigma\boxtimes \tau)$ to $C^{\infty}_c(P\widetilde{\widetilde{R}}, \sigma\boxtimes \tau)$, we then also have a non-zero map:
\[  C^{\infty}_c(P\widetilde{\widetilde{R}}, \sigma\boxtimes \tau)_{R', \psi'} \rightarrow \pi_{R', \psi'} .
\]
But, $C^{\infty}_c(P\widetilde{R}, \sigma\boxtimes \tau)_{R', \psi'}$ is isomorphic to a subspace of $\sigma \times {}^-\tau$ (here we have to use that the level of $\tau$ is also $l$) and the second one is isomorphic to ${}^-\pi$. This gives the requied non-zero map and so implies the corollary.
\end{proof}

The above example can be regarded as an instance on how the two operations -- Bernstein-Zelevinsky derivatives and parabolic inductions -- commute. More extensive studies and applications appear in \cite{Ch22+b} and \cite{Ch22+}. One may also see Corollary \ref{cor app on construction of deriv} below for a consequence. While our above application is stated for general linear groups, the statement for Corollary \ref{cor embedding from second adjointness} is more general, and so one may hope to apply similar idea to other situations.

\subsection{Connection to commutative triples and its dual version}
We continue to work on $G_n=\mathrm{GL}_n(F)$ in this subsection.

\begin{definition}
A {\it segment} is a combinatorial datum of the form $[a,b]_{\rho}$ for some cuspidal representation $\rho$ of some $G_k$ and for some $a, b \in \mathbb Z$ with $b-a \in \mathbb Z_{\geq 0}$. 
\end{definition}
To each segment $\Delta$, we associate an essentially square-integrable representation denoted by $\mathrm{St}(\Delta)$ \cite{Ze80} (see \cite{Ch22+c} for more details of the notion). We now introduce a notion of derivatives, which can be seen as a computable replacement for some structure of Bernstein-Zelevinsky derivatives (see \cite{Ch22+c}).

\begin{definition} \label{def unique embedd derivative}
Let $\Delta$ be a segment. Let $\pi$ be an irreducible smooth representation of $G_n$.
\begin{itemize}
\item  Define $D_{\Delta}(\pi)$ to be the unique (if exists) irreducible smooth representation of some $G_k$ such that 
\[  \pi \hookrightarrow D_{\Delta}(\pi) \times \mathrm{St}(\Delta) .
\]
Set $D_{\Delta}(\pi)=0$ if such module does not exist.
\item Define $I_{\Delta}(\pi)$ to be the unique irreducible subrepresentation of $\mathrm{St}(\Delta)\times \pi$. 
\end{itemize}
We also remark that when $D_{\Delta}(\pi)\neq 0$, Frobenius reciprocity also gives that 
\[   \pi_N \twoheadrightarrow D_{\Delta}(\pi) \boxtimes \mathrm{St}(\Delta) ,
\]
where $N$ is the unipotent radical of the parabolic subgroup for $D_{\Delta}(\pi)\times \mathrm{St}(\Delta)$. It is also a standard fact that if $D_{\Delta}(\pi)\neq 0$, there is a unique embedding:
\[    D_{\Delta}(\pi)\boxtimes \mathrm{St}(\Delta) \hookrightarrow \pi_N .
\]
\end{definition}

\begin{definition} \label{def commutative triple} \cite{Ch22+b}
Let $\rho_1$ and $\rho_2$ be cuspidal representations of $G_{n_1}$ and $G_{n_2}$ respectively. Let $\Delta_1=[a_1,b_1]_{\rho_1}, \Delta_2=[a_2,b_2]_{\rho_2}$ be segments. Let $l_k=(b_k-a_k+1)n_k$ ($k=1,2$). Let $\pi$ be an irreducible representation of $G_n$. Let $P_1=P_{l_2+n-l_1,l_1}$, a subgroup of $G_{l_2+n}$ defined in Section \ref{ss appl derivatives}, and similarly $P_2=P_{l_2, n}$. Let $N_1$ be the unipotent radical of $P_1$. We say that $(\Delta_1, \Delta_2, \pi)$ is a {\it pre-commutative triple} if the following composition:
\[      D_{\Delta_1}\circ I_{\Delta_2}(\pi) \boxtimes \mathrm{St}(\Delta_1) \hookrightarrow  (I_{\Delta_2}(\pi))_{N_{1}} \hookrightarrow  (\mathrm{St}(\Delta_2) \times \pi)_{N_{1}} \twoheadrightarrow C^{\infty}_c(P_{2}P_{1}, \mathrm{St}(\Delta_2)\boxtimes \pi)_{N_{1}} ,
\] 
where 
\begin{itemize}
\item the first map is the unique embedding from the remark in Definition \ref{def unique embedd derivative};
\item the second map is induced from the unique embedding $I_{\Delta_2}(\pi) \hookrightarrow \mathrm{St}(\Delta_2)\times \pi$ in Definition \ref{def unique embedd derivative};
\item the third map is induced from the natural surjection from $\mathrm{St}(\Delta_2) \times \pi$ to $C^{\infty}_c(P_{2}P_{1}, \mathrm{St}(\Delta_2)\boxtimes \pi)$ in the geometric lemma.
\end{itemize}

\end{definition}

We now define a dual version:

\begin{definition} \label{def dual derivative integral}
Let $\Delta$ be a segment. Let $\pi$ be an irreducible representation of $G_n$. 
\begin{itemize}
 \item Define $D^{\vee}_{\Delta}(\pi)$ to be the unique (if exists) irreducible representation such that 
\[  D^{\vee}_{\Delta}(\pi) \times\mathrm{St}(\Delta)  \twoheadrightarrow \pi.  \]
Again we set $D^{\vee}_{\Delta}(\pi)=0$ if such representation does not exist.
 \item Define $I^{\vee}_{\Delta}(\pi)$ to be the unique irreducible quotient of $\mathrm{St}(\Delta) \times \pi$.
\end{itemize}
\end{definition}

\begin{definition} \label{def dual commutative triple}
We use the notations in Definition \ref{def commutative triple}. Let $N_{1}^-$ be the unipotent radical of $P_{1}^-$. We say that $(\Delta_1, \Delta_2, \pi)$ is a {\it dual commutative triple} if the following composition:
\[   C^{\infty}_c(P_{2}P_{1}^-, \mathrm{St}(\Delta_2)\boxtimes \pi)_{N_{1}^-}\hookrightarrow     (\mathrm{St}(\Delta_2)\times \pi)_{N_{1}^-} \twoheadrightarrow    I^{\vee}_{\Delta_2}(\pi)_{N_{1}^-} \rightarrow D^{\vee}_{\Delta_1}\circ I^{\vee}_{\Delta_2}(\pi) \boxtimes \mathrm{St}(\Delta_2) 
\]
is non-zero, where 
\begin{itemize}
\item the first map is induced from the natural embedding from $C^{\infty}_c(P_{2}P_{1}^-, \mathrm{St}(\Delta_2)\boxtimes \pi)$ to $\mathrm{St}(\Delta_2)\times \pi$;
\item the second map is induced from the natural map from $\mathrm{St}(\Delta)\times \pi$ to $I^{\vee}_{\Delta}(\pi)$;
\item the third map is the map from the first bullet of Definition \ref{def dual derivative integral}.
\end{itemize}
\end{definition}

\begin{proposition} \label{prop dual commutative}
Let $\Delta_1, \Delta_2$ be segments. Let $\pi$ be an irreducible representation of $G_n$. Then $(\Delta_1, \Delta_2, \pi)$ is a commutative triple if and only if $(\Delta_1^{\vee}, \Delta_2^{\vee}, \pi^{\vee})$ is a dual commutative triple.
\end{proposition}

\begin{proof}
This follows from Theorem \ref{thm filtrations coincide bc pairing} and the following standard facts: for a segment $\Delta$, and for irreducible representations $\pi$ and $\pi'$, 
\begin{enumerate}
\item $(\pi \times \pi')^{\vee} \cong \pi^{\vee}\times \pi'{}^{\vee}$;
\item $I_{\Delta}(\pi)^{\vee} \cong I^{\vee}_{\Delta^{\vee}}(\pi^{\vee})$;
\item $D_{\Delta}(\pi)^{\vee} \cong D^{\vee}_{\Delta^{\vee}}(\pi^{\vee})$.
\end{enumerate}
We remark that (2) and (3) can also be deduced from (1) in a quite straightforward manner,
\end{proof}

Proposition \ref{prop dual commutative} implies that the notion of dual commutative triples coincides with another notion called LdRi-commutative triples in \cite{Ch22+b}.

\subsection{More on commutativity of derivatives}

An irreducible smooth representation $\pi$ of $G_n$ is said to be {\it thickened} if $\mathrm{lev}(\pi)=\mathrm{lev}({}^-\pi)$. An application of Corollary \ref{cor descending derivatives} is a construction of derivatives under the highest Bernstein-Zelevinsky derivative.

\begin{corollary} \label{cor app on construction of deriv}
Let $\pi \in \mathrm{Irr}(G_n)$ be thickened. Let $\Delta$ be a segment such that $D_{\Delta}(\pi)\neq 0$. Let $\tau=D_{\Delta}(\pi)$. Then ${}^-\tau=D_{\Delta}({}^-\pi)$.
\end{corollary}

\begin{proof}
Here we use a standard fact that if $\pi$ is thickened, then the level of $D_{\Delta}(\pi))$ is equal to the level of $\pi$ (see \cite[Lemma 3.11 and Proposition 1.2]{Ch22+c}). Another standard fact also provides that $I^{\vee}_{\Delta}\circ D_{\Delta}(\pi)\cong \pi$. Hence, $\pi$ is the irreducible quotient of $\mathrm{St}(\Delta)\times \tau$. Now Corollary \ref{cor descending derivatives} implies that ${}^-\pi$ is the irreducible quotient of $\mathrm{St}(\Delta) \times {}^-\tau$ i.e. $I^{\vee}_{\Delta}({}^-\tau)\cong {}^-\pi$ and so ${}^-\tau=D_{\Delta}({}^-\pi)$.
\end{proof}

\end{document}